\documentclass[11pt]{article}
\usepackage{amsmath,amssymb,amsthm,amscd}

\def \N {\mathbb{N}}
\def \Z {\mathbb{Z}}
\def \Q {\mathbb{Q}}
\def \a {\mathbf{a}}
\def \b {\mathbf{b}}
\def \c {\mathbf{c}}
\def \p {\mathbf{p}}
\def \q {\mathbf{q}}
\def \u {\mathbf{u}}
\def \v {\mathbf{v}}
\def \w {\mathbf{w}}
\def \x {\mathbf{x}}
\def \y {\mathbf{y}}
\def \C {\mathbf{C}}
\def \P {\mathbf{P}}
\def \L {\mathcal{L}}

\def \S {\mathcal{S}}
\def \CF {{\cal CF}}
\def \vp {\bar{v}_p}
\def \BS {\operatorname{BS}}

\newtheorem{theorem}{Theorem}[section]
\newtheorem{proposition}[theorem]{Proposition}
\newtheorem{lemma}[theorem]{Lemma}
\newtheorem{corollary}[theorem]{Corollary}
\theoremstyle{plain}\newtheorem{observation}[theorem]{Observation}
\newtheorem{conjecture}[theorem]{Conjecture}

\numberwithin{theorem}{section}

\title{Groups with poly-context-free word problem}

\author{Tara Brough}
\date{April 10, 2011}

\begin{document}

\maketitle

\begin{abstract}
We consider the class of groups whose word problem is \emph{poly-context-free};
that is, an intersection of finitely many context-free languages.  

We show that any group which is virtually a finitely generated subgroup of a direct
product of free groups has poly-context-free word problem, and
conjecture that the converse also holds.  We prove our conjecture for several
classes of soluble groups, including metabelian groups and torsion-free soluble
groups, and present progress towards resolving the conjecture for soluble groups
in general.

Some of the techniques introduced for proving languages not to be poly-context-free may 
be of independent interest.
\end{abstract}

\section{Introduction}

The \emph{word problem} of a group $G$ with respect to a finite generating set $X$,
denoted $W(G,X)$, is 
the set of all words in elements of $X$ and their inverses which represent the 
identity element of $G$.  A \emph{(formal) language} is a set of words over some finite
alphabet, so $W(G,X)$ can be considered as a language. 

The study of word problems of groups as languages has developed slowly since 
the beginnings of language theory in the 1950s.  
In 1971, Anisimov \cite{Ani}
published a proof that a group has regular word problem if and only if it is finite. 
The first really significant development in the area was the classification of the groups 
with context-free word problem by Muller and Schupp in the 1980s \cite{MulSch1, MulSch2, Dun}:  
a finitely generated group has context-free word problem if and only if it is virtually free.
Since then, research activity in this area has increased, and groups with word problem in 
various other language classes, generally somewhat related to the context-free languages, 
have been studied, for example in \cite{Herbst, HRRT, HRees, HR, HR2, HOT, HRS, KO, LehSch, Shap}.
The general aim is to determine what implications the language type of a group's 
word problem has for the structure of the group and vice versa.

One natural class of languages to consider is the closure of the context-free languages 
under intersection.  
Some research has been done on this class (see for example \cite{LW}, \cite{Wot} and \cite{Gorun}), 
but it does not appear to have a consistent name.  We call a language \emph{$k$-context-free} 
(henceforth abbreviated to $k$-$\CF$)
if it is an intersection of finitely many context-free languages, and 
\emph{poly-context-free (poly-$\CF$)} if it is $k$-$\CF$ for some $k\in \N$.

This paper is concerned with the class of poly-$\CF$ groups.  A group is said to be 
\emph{poly-$\CF$} if its word problem is a poly-$\CF$ language.  
The property of being poly-$\CF$ is independent of the choice of finite generating set,
and the class of poly-$\CF$ groups is closed under taking finitely generated subgroups,
finite index overgroups, and finite direct products.  
All but the last of these properties are typical of classes of groups defined by the 
language type of their word problem.

A general classification of these groups appears to be hard.
However, we prove a result (Theorem~\ref{polycf-sol}) which comes close in the case of
soluble groups.

We conjecture that the only poly-$\CF$ groups are those obtained from virtually free groups
using the above-mentioned operations; that is, that a group is poly-$\CF$ if and only if
it is virtually a finitely generated subgroup of a direct product of free groups 
(Conjecture~\ref{solconj}).
This would mean that the only soluble poly-$\CF$ groups are the virtually abelian groups.
Theorem~\ref{polycf-sol} gives substantial evidence towards this special case of our conjecture.

In \cite{HRRT}, the co$\CF$ groups (groups whose word problem is the complement of a 
context-free language) were studied.  Various closure properties of the co$\CF$ groups
were determined, most of which carry over easily to 
poly-$\CF$ groups (see Proposition~\ref{closureprops} below).
Additionally, several classes of groups were shown not to be co$\CF$, using a method
\cite[Proposition~14]{HRRT} based on the correspondence between context-free 
languages and semilinear sets (see Section~\ref{semisec} below).  
We prove a strengthened version of \cite[Proposition~14]{HRRT}
(see Proposition~\ref{prop} below), which enables us to deduce that any group 
proved not co$\CF$ using \cite[Proposition~14]{HRRT} is also not poly-$\CF$.
Some examples of such groups are finitely generated nilpotent or polycyclic 
groups that are not virtually abelian.  It was these results that led to the 
attempt at a characterisation of the soluble poly-$\CF$ groups.

A major open problem for co$\CF$ groups is whether they are closed under taking free products.
It was suggested by Derek Holt that closure under free products might be much easier
to determine for poly-$\CF$ groups, but so far this problem also remains open,
though we believe that the word problem of $\Z^2 *\Z$ is not poly-$\CF$.
The poly-$\CF$ groups are somewhat related to the co$\CF$ groups, in the sense that 
if our main conjecture is true, then the poly-$\CF$ groups are a subclass of the co$\CF$-groups,
as we explain in Section~\ref{group}, following Conjecture~\ref{polyconj}.\\

Our main tools are introduced in Section~\ref{background}.  These are: various 
closure properties of the classes of poly-$\CF$ languages and poly-$\CF$ groups;
the relationship between bounded context-free languages and semilinear sets, due to 
Parikh \cite{Par} and Ginsburg and Spanier \cite{GinSpa2}; and a result by the 
author and Derek Holt \cite{BH}, showing that every finitely generated soluble 
group that is not virtually abelian has a subgroup isomorphic 
to one of a small number of types.

In Section~\ref{language}, we study the class of poly-$\CF$ languages, with a particular
focus on methods for proving languages to be \emph{not} poly-$\CF$.  To this end,
we develop several tools based on the correspondence between context-free languages
and stratified semilinear sets introduced in Section~\ref{semisec}.  
In Corollary~\ref{parikh}, we show that a language satisfying certain properties is neither 
poly-$\CF$ nor co$\CF$, while Theorem~\ref{L(n,k)} exhibits sequences of languages $L^{(n,k)}$, 
where $n, k\in \N$, such that for all $n$, the 
language $L^{(n,k)}$ is an intersection of $k$ but not $k-1$ context-free languages.  
This is an extension of a result by Liu and Weiner \cite{LW}.
 
In Section~\ref{group}, we present the known examples of poly-$\CF$ groups, and 
conjecture that these are the only ones.  We give some evidence for this conjecture
(Conjecture~\ref{polyconj}), in 
the form of results showing that it holds in the classes of nilpotent, Baumslag-Solitar
and polycyclic groups, and for the groups $G(\c)$ introduced in \cite{BH}, which are 
also shown to be not co$\CF$ if they are not virtually abelian.

We conclude with a section applying the results of Section~\ref{group} and \cite{BH} to
prove the metabelian and torsion-free soluble cases of our conjecture, and to narrow 
down the possibilities for which soluble groups could be poly-$\CF$.

\section{Background and notation}\label{background}

\subsection{Notation}

$\N, \Z$ and $\Q$ denote the natural numbers, integers and rationals respectively.
We denote the natural numbers with zero included by $\N_0$.

For $r\in \N$ and $1\leq i\leq r$, the vector in $\N_0^r$ with a $1$ in the $i$-th position 
and zeroes elsewhere will be denoted by $e_i$.  With the exception of these, 
all vectors are represented by bold letters.  
We denote the $i$-th component of the vector $\v$ by $\v(i)$.

For a set $X$, we denote the \emph{Kleene star closure} of $X$, which is the set of all 
finite length strings (also called words) of elements of $X$, by $X^*$.
In the special case $X = \{x\}$, we often denote $X^*$ by $x^*$.

\subsection{Closure properties of the poly-$\CF$ languages}

Many closure properties of the classes of $k$-${\cal CF}$ and poly-$\CF$ languages
can be deduced from the similar properties for context-free languages;
for details of these, see (for example) \cite{HopUll}.   

\begin{proposition}\label{polyclose}
For any $k\in \N$, the class of $k$-$\cal{CF}$ languages is closed under
inverse homomorphisms, inverse generalised sequential machine mappings,
union with context-free languages and intersection with regular languages.
The class of poly-$\cal{CF}$ languages is closed under all these operations,
and also under intersection and union.
\end{proposition}
\begin{proof}
Let $L = L_1\cap\ldots\cap L_k$ with each $L_i$ context-free and let $\Sigma$
be the alphabet of $L$.
Let $\Gamma$ be an alphabet and let $\phi$ be a homomorphism from $\Gamma^*$ to $\Sigma^*$, or a
generalised sequential machine mapping with input alphabet $\Gamma$ and output alphabet $\Sigma$.  Then
\begin{align*}
\phi^{-1}(L) &= \{ w\in \Gamma^* \mid \phi(w)\in L_i \; (1\leq i\leq k)\}\\
&= \bigcap_{i=1}^k \{ w\in \Gamma^* \mid \phi(w)\in L_i\} = \bigcap_{i=1}^k \phi^{-1}(L_i),
\end{align*}
and so, since the class of context-free languages is closed under inverse homomorphisms
and inverse generalised sequential machine mappings, $\phi^{-1}(L)$ is $k$-${\cal CF}$.

The class of context-free languages is closed under union and under intersection with 
regular languages.  Thus if $R$ is regular, then
$L\cap R = L_1\cap\ldots\cap L_{k-1}\cap (L_k\cap R)$ is $k$-${\cal CF}$;
and if $M$ is context-free, then
$L\cup M = \bigcap_{i=1}^k (L_i\cup M)$ is $k$-$\cal{CF}$.

The closure of the class of poly-$\cal{CF}$ languages under intersection is obvious,
since if $L_1$ is $k_1$-$\cal{CF}$ and $L_2$ is $k_2$-$\cal{CF}$, then
$L_1\cap L_2$ is an intersection of $k_1 + k_2$ context-free languages.

If $L = \cap_{i=1}^m L_i$ and $M = \cap_{j=1}^n M_j$, with each $L_i$ and $M_j$
context-free, then
\[ L\cup M = \left(\bigcap_{i=1}^m L_i\right)\cup\left(\bigcap_{j=1}^n M_j\right) = 
\bigcap_{i=1}^m\bigcap_{j=1}^n (L_i\cup L_j) \]
is $mn$-${\cal CF}$, so the class of poly-${\cal CF}$ languages is also closed under union.
\end{proof}

The closure of the poly-$\CF$ languages under union and intersection was already observed
by Wotschke \cite{Wot}, who also showed, using a theorem of Liu and Weiner 
(see Section~\ref{L(k) section} below), that
the poly-$\CF$ languages are not closed under complementation and are thus properly contained 
in the Boolean closure of the context-free languages \cite[Theorem~II.4]{Wot} .

Any recursively enumerable language can be expressed as a homomorphic image of the intersection of two 
deterministic context-free languages \cite{GGH}. 
Every poly-$\CF$ languages is context-sensitive,
since the context-sensitive languages are closed under intersection and contain the context-free languages. 
Thus the poly-$\CF$ languages are not closed under homomorphisms.

\subsection{Basic properties of the poly-$\CF$ groups} 

A central result in the theory of word problems of groups as languages is the following,
for which a proof is given in \cite{HRRT}. 
We denote the complement of $W(G,X)$ in $X^*$ by co$W(G,X)$.

\begin{lemma}\label{genset}{\rm \cite[Lemma~1]{HRRT}}
Let $\mathcal{C}$ be a class of languages closed under inverse homomorphisms 
and let $G$ be a finitely generated group.  Then the following hold.
\begin{enumerate}
\item $W(G,X)\in \mathcal{C}$ for some finite generating set $X$ if and only if 
for every finite generating set $Y$, $W(G,Y)\in \mathcal{C}$.
\item co$W(G,X)\in \mathcal{C}$ for some finite generating set $X$ if and only if 
for every finite generating set $Y$, co$W(G,Y)\in \mathcal{C}$.
\end{enumerate}
\end{lemma}

In this case, we call $G$ a \emph{$\cal{C}$ group} if $W(G)$ is in $\cal{C}$, and a \emph{co${\cal C}$ group} 
if co$W(G)$ is in $\cal{C}$, and say that $\cal{C}$ groups or co${\cal C}$ groups are
\emph{insensitive} to choice of generators.

\begin{lemma}{\rm \cite[Lemma 2]{HRRT}}\label{fgsub}
Let $\cal{C}$ be a class of languages closed under inverse homomorphisms and 
intersection with regular sets.  Then the classes of $\cal{C}$ groups and co$\cal{C}$ groups
are closed under taking finitely generated subgroups.
\end{lemma}

\begin{lemma}{\rm \cite[Lemma 5]{HRRT}}\label{fiover}
Let $\cal{C}$ be a class of languages closed under union with regular sets and 
inverse generalised sequential machine mappings.  Then the classes of $\cal{C}$ groups
and co$\cal{C}$ groups are closed under passing to finite index overgroups.
\end{lemma}

Thus, by Proposition~\ref{polyclose}, we have:

\begin{proposition}\label{closureprops}
The classes co$\CF$ and $k$-$\CF$ groups (for any $k\in \N$) are
insensitive to choice of generators and closed under passing to
finitely generated subgroups and passing to finite index overgroups.
\end{proposition}

\subsection{Semilinear sets}

A useful tool for proving languages not to be poly-$\CF$ is a relationship 
between context-free languages and semilinear sets, introduced by 
Parikh \cite{Par} and then strengthened, in the case of bounded languages, 
by Ginsburg and Spanier \cite{GinSpa2}.  

A \emph{linear set} is a subset $L$ of $\N_0^r$ for which there exist a 
\emph{constant vector} $\c\in \N_0^r$ 
and a finite set of \emph{periods} $P = \{\p_i \mid 1\leq i\leq n\}\subseteq \N_0^r$ such that 
\[L = \{\c + \sum_{i=1}^n \alpha_i \p_i \mid \alpha_i\in \N_0\}.\]  
Note that the set of periods $P$ is not uniquely determined.
A \emph{semilinear set} is a union of finitely many linear sets.

Following Ginsburg~\cite{Gin}, we will use the notation $L(\c; \p_1,\ldots,\p_n)$, or $L(\c; P)$, 
for a linear set with constant $\c$ and set of periods $P = \{\p_1,\ldots,\p_n\}$.  
For $C$ a set of constant vectors, we will denote $\bigcup_{\c\in C} L(\c; P)$ by $L(C; P)$. 
If $C = \{\c_1,\ldots,\c_m\}$, we will also write $L(\c_1,\ldots,\c_m;\p_1,\ldots,\p_n)$
for $L(C;P)$.

If $L = L(\c; P)$, we define $L^{\Q}$ to be the set $\{\c + \sum_{i=1}^n a_i\p_i \mid a_i\in \Q\}$.  
This is a coset in $\Q^n$ of the $\Q$-subspace spanned by $P$.  
We define $L^{\mathbf{0}}$ to be $L(\mathbf{0}; P)$, that is, the linear set having the same periods
as $L$ and constant $\mathbf{0}$.\\

A subset $P$ of $\N_0^r$ is \emph{stratified} if it satisfies the following conditions:
\begin{enumerate}
\item each $\p\in P$ has at most two non-zero components, and 
\item there do not exist $i<j<k<l$ and non-zero $a, b, c, d\in \N$ such that 
$ae_i + be_k$ and $ce_j+de_l$ are both in $P$.  
\end{enumerate}

A linear set is \emph{stratified} if it can be expressed using a stratified set of periods.
A semilinear set is \emph{stratified} if it can be expressed as a union of finitely many 
stratified linear sets.  (We follow Liu and Weiner \cite{LW} for this terminology.)  
Note that stratified linear and semilinear sets are not generally 
stratified sets in the sense of the previous paragraph. 

\subsubsection{Stratified semilinear sets and bounded poly-$\CF$ languages}\label{semisec}

The \emph{commutative image} of a language $L$ over $\{a_1,\ldots,a_r\}$ is 
the subset of $\N_0^r$ given by mapping each $w\in L$ to 
the tuple $(n_1,\ldots,n_r)$, where $n_i$ is the number of occurrences of $a_i$ in $w$.
Parikh's theorem \cite{Par} says that the commutative image of a context-free 
language is always a semilinear set.
The converse of Parikh's theorem does not hold:  consider for example the language
$\{a^m b^n c^m d^n \mid m,n\in \N_0\}$.

A language $L\subseteq X^*$ is \emph{bounded} if there exist $w_1,\ldots,w_n\in X^*$ 
such that $L\subseteq w_1^*\ldots w_n^*$, in which case we can define
a corresponding subset of $\N_0^n$:  
\[ \Phi(L) = \{ (m_1,\ldots,m_n) \mid m_i\in \N_0, w_1^{m_1}\ldots w_n^{m_n}\in L\}.\]
When $w_1,\ldots,w_n$ are distinct single symbols, this is the same as the commutative
image of $L$.  Thus the following result of Ginsburg and Spanier strengthens Parikh's theorem in the case of 
bounded languages.
 
\begin{theorem}\label{Parikh}{\rm \cite[Theorem 5.4.2]{Gin}}
Let $W\subseteq w_1^*\ldots w_n^*$, each $w_i$ a word.  
Then $W$ is context-free if and only if $\Phi(W)$ is a stratified semilinear set. 
\end{theorem}

Ginsburg and Spanier used different notation, which made it more transparent 
how to get from $\Phi(W)$ back to $W$.  But as we will only require the `only if' direction, 
we prefer this tidier notation.

Theorem~\ref{Parikh} is easily extended to the poly-$\CF$ languages.

\begin{corollary}\label{polyParikh}
If $L$ is a $k$-$\CF$ language, then for any $w_1,\ldots,w_n$, 
the subset ${\Phi(L\cap w_1^*\ldots w_n^*)}$ of $\N_0^n$ is an 
intersection of $k$ stratified semilinear sets.
\end{corollary}
\begin{proof}
Let $L = L_1\cap\ldots\cap L_k$ with each $L_i$ context-free, and let $W = w_1^*\ldots w_n^*$, 
where each $w_i$ is a word in the alphabet of $L$. 
For $1\leq i\leq k$, let $M_i = L_i\cap W$.  Then $L\cap W = L_1\cap\ldots\cap L_k\cap W = \bigcap_{i=1}^k M_i$ and 
\begin{align*}
\Phi(L\cap W) &= \{ (m_1,\ldots,m_n) \mid m_i\in \N_0, w_1^{m_1}\ldots w_n^{m_n}\in L\cap W\}\\
        &= \bigcap_{i=1}^k \{ (m_1,\ldots,m_n) \mid m_i\in \N_0, w_1^{m_1}\ldots w_n^{m_n}\in M_i\}\\
        &= \bigcap_{i=1}^k \Phi(M_i)
\end{align*}
and each $\Phi(M_i)$ is a stratified semilinear set by Theorem~\ref{Parikh}.
\end{proof}

Proving that a given semilinear set is not stratified is by no means straightforward, since there 
can be many different ways of expressing a semilinear set as a union of finitely many linear sets.  
Ginsburg \cite{Gin} mentioned that there was no known decision procedure for determining whether an
arbitrary semilinear set is stratified, and it appears that this is still an open problem. 

\subsubsection{Closure properties of the class of semilinear sets}

The class of semilinear sets is obviously closed under union.
Thinking geometrically, one would also expect this class to be closed under the 
other Boolean operations.  This is indeed true, but much less easy to show.

The intersection of finitely many linear sets is always a 
semilinear set of quite a restricted form.  
This result can be derived from the proof of Theorem~5.6.1 in \cite{Gin}. 
For another proof, obtained independently by the author, see \cite[Proposition~2.3]{TB}.

\begin{proposition}\label{intlin}
If $L$ is the nonempty intersection of linear subsets $L_1,\ldots,L_n$ of $\N_0^r$, 
then $L$ is semilinear.    Moreover, 
\[ L = L(\C_1,\ldots,\C_k; \P_1,\ldots,\P_m),\]
where $\C_i\in\N_0^r$, and $\P_1,\ldots,\P_m$ are such that 
$\bigcap_{i=1}^n L_i^{\mathbf{0}} = L(\mathbf{0}; \P_1,\ldots,\P_m)$.\\
If $L_1,\ldots,L_n$ all have constant vector zero, then $L$ is linear with constant vector zero.
\end{proposition}

\begin{corollary}\label{intsl}{\rm \cite[Theorem~5.6.1]{Gin}}
Let $L$ be an intersection of finitely many semilinear sets.  
Then $L$ is a semilinear set.
\end{corollary}

\begin{proposition} {\rm \cite[Theorem~6.2 and Corollary~1]{GinSpa1} }\label{compsl}
If $L$ and $M$ are semilinear subsets of $\N_0^r$, then $M - L$ is also a semilinear subset
of $\N_0^r$ and effectively calculable from $L$ and $M$.  In particular, 
since $\N_0^r$ is semilinear, if $L$ is a 
semilinear subset of $\N_0^r$, then the complement of $L$ in $\N_0^r$ is semilinear.
\end{proposition}

\subsubsection{Dimension of linear sets}

If $V$ is a subspace of a vector space $W$ with $\dim(V) < \dim(W)$, 
then the \emph{dimension} of a coset of $V$ in $W$ is 
defined to be the dimension of $V$.
The \emph{dimension} of a linear set $L$ is defined to be the dimension of $L^{\Q}$ or, equivalently,
the dimension of the vector space over $\Q$ spanned by the periods of $L$.

We record here a result about the dimension of linear sets which will be useful later.  
This is a known result, but the only reference we have for it is \cite{LW}, where the proof
given is incorrect.  A proof is included in the author's Ph.D. thesis \cite[Proposition~2.10]{TB}.

\begin{proposition}\label{dimlin}
A linear set of dimension $n+1$ cannot be expressed as a union of finitely many linear sets of dimension $n$ or less.
\end{proposition}

\subsection{Subgroups of finitely generated soluble groups}

The following theorem is a combination of Theorems~3.3 and~5.2 in \cite{BH}.
By $\Z^\infty$, we mean the free abelian group of countably infinite rank.
For the definition of a proper Gc-group, see Section~\ref{Gc-sec}.
A group is \emph{metabelian} if it has derived length at most $2$.

\begin{theorem}\label{solsubs}
Let $G$ be a finitely generated soluble group which is not virtually abelian.  
Then $G$ has a subgroup isomorphic to at least one of the following.
\begin{enumerate}
\item $\Z^\infty$;
\item a proper Gc-group; 
\item a finitely generated group $H$ with an infinite
normal torsion subgroup $U$, such that $H/U$ is
either free abelian or a proper Gc-group.
\end{enumerate}
If $G$ is metabelian, then the subgroup $H$ in (iii) can always be taken to be $C_p\wr \Z$
for some prime $p$.
\end{theorem}

Since the class of poly-$\CF$ groups
is closed under taking finitely generated subgroups, this gives a very useful approach towards
resolving our conjecture for soluble groups.

\section{Poly-$\CF$ languages}\label{language}

Recall that a $k$-$\cal{CF}$ language is an intersection of $k$ context-free languages, 
and a poly-$\cal{CF}$ language is a language which is $k$-$\cal{CF}$ for some $k\in \N$.
In this section, we shall primarily be concerned with proving some results which will assist us 
in determining that the word problems of certain groups are not poly-$\cal{CF}$.

\subsection{A criterion for a language to be neither poly-$\CF$ nor co$\CF$}

In \cite[Proposition~14]{HRRT}, a technique was developed for proving a subset of 
$\N_0^r$ not to be the complement of a semilinear set.  This was used in combination
with Parikh's theorem to prove various classes of groups not to be co$\CF$.
The proof is fairly long and technical.  The authors were presumably unaware of the fact
that the complement of a semilinear set is semilinear.  This fact allows us to give a 
much simpler proof of their result, and to strengthen it.
 
If $\a$ and $\b$ are vectors in $\N_0^r$ and $\N_0^s$ respectively, then we denote
by $(\a;\b)$ the vector in $\N_0^{r+s}$ which consists of all the components of $\a$ in order,
followed by those of $\b$ in order.  
When talking about vectors in $\N_0^{r+s}$, if we write $(\a;\b)$, then it is understood that
$\a\in \N_0^r$ and $\b\in \N_0^s$.
For $\a\in \N_0^r$, we define $\sigma(\a) = \sum_{i=1}^r \a(i)$.

We use the following lemma, extracted from the proof of Proposition 11 in \cite{HRRT}.
We call a vector $\v\in \N_0^{r+s}$ {\em simple} if its first $r$ components are all zero, 
and {\em complex} otherwise.  
The proof is quoted from \cite{HRRT} with only minor
modifications.

\begin{lemma}\label{C_L}
Let $L = L_1\cup\ldots\cup L_n$, with each $L_i$ a linear subset of $\N_0^{r+s}$.
Then there exists a constant $C\in \N$ such that if $(\a;\b)\in L$ can be expressed using 
only complex periods, then $\b(j) < C\sigma(\a)$ for all $1\leq j\leq s$.
\end{lemma}
\begin{proof}
Fix some $i\in \{1,\ldots,n\}$ and let $L_i = L(\c_i; P_i)$.  
If $(\p;\q)\in P_i$ is a complex period, then $\sigma(\p)\neq 0$, 
so there exists $t$ such that $\q(j) < t\sigma(\p)$ for ${1\leq j\leq s}$.  
Since $P_i$ is finite, we can choose the same $t$ for all $(\p;\q)\in P_i$.
If, for $k=1,2$, $(\a_k;\b_k)\in \N_0^{r+s}$ satisfy $\b_k(j) < t\sigma(\a_k)$,
then $(\b_1 + \b_2)(j) < t\sigma(\a_1 + \a_2).$  
Thus there is a constant $q\in \N_0$, which
can be taken to be ${\max\{\c_i(j) \mid 1\leq j\leq r\}}$, such that if $(\a;\b)\in L_i$ can be 
expressed using only complex periods, then $\b(j) < t\sigma(\a) + q$ for all $1\leq j\leq s$.

Now let $C\in \N$ be twice the maximum of all of the constants $t,q$ that arise for all $L_i$.
Then, for any $(\a;\b)\in L$ which can be expressed using only complex periods, $\b(j) < C\sigma(\a)$ 
for all $1\leq j\leq s$.
\end{proof}

\begin{proposition}\label{prop}
Let $L\subseteq \N_0^{r+s}$ for some $r, s\in \N$.  Let $f:\N\rightarrow \N$ be an unbounded function 
and suppose that, for every $k\in \N$, there exists ${\a\in \N_0^r\setminus \{\mathbf{0}\}}$ such that 
the following hold:
\begin{enumerate}
\item There exists $\b\in \N_0^s$ such that $(\a;\b)\in L$.
\item If $(\a;\b)\in L$ then $\b(j) \geq k\sigma(\a)$ for some $1\leq j\leq s$.
\item If $(\a;\b), (\a;\b')\in L$ with $\b\neq \b'$, then $|\b(l) - \b'(l)|\geq f(k)$
for some $1\leq l\leq s$.
\end{enumerate}
Then $L$ is not a semilinear set. 
\end{proposition}
\begin{proof}
Let $L$ be as in the statement of the proposition and suppose that $L = \bigcup_{i=1}^n L_i$, where 
each $L_i = L(\c_i;P_i)$ is a linear subset of $\N_0^{r+s}$.
By Lemma~\ref{C_L}, there exists a constant $C\in \N$ such that if $(\a;\b)\in L$ can be expressed using 
only complex periods in some $L_i$, then $\b(j) < C\sigma(\a)$ for all $1\leq j\leq s$.  

Choose $k>C$, and suppose $\a$ satisfies the hypotheses of the proposition with respect to $k$. 
If $(\a;\b)\in L$, then $(\a;\b)$ cannot be expressed using only complex periods, so some $P_i$
must contain a simple period $(\mathbf{0};\v)$ with $\v$ non-zero.  But then 
$(\a;\b+\v)\in L_i\subseteq L$ and so, for some $1\leq l\leq s$, 
\[ |\v(l)| = |(\b + \v)(l) - \b(l) |\geq f(k).\]
So for all $k>C$, there is a non-zero simple period $\v_k$ in $\cup_{i=1}^n P_i$,
with some component of $\v_k$ being at least $f(k)$.  But since $\cup_{i=1}^n P_i$ is finite
and $f(k)$ is unbounded, this is impossible.  Thus $L$ is not a semilinear set.
\end{proof}

In \cite[Proposition~14]{HRRT}, instead of our condition (i), it is required that there 
is a \emph{unique} $\b\in \N_0^r$ such that $(\a;\b)\in L$ (and thus there is no condition (iii)
or mention of the unbounded function $f$); instead of our condition (ii), it is required that 
$\b(j) \geq k\sigma(\a)$ for \emph{every} $1\leq j\leq s$.  
The conclusion is that $L$ is not the complement of a semilinear set.

Our hypothesis is considerably weaker, and the conclusion is equally strong, since the complement of a 
semilinear set is semilinear.\\

For $\v = (n_1,\ldots,n_r)\in \N_0^r$ and $\tau$ a permutation of $\{1,\ldots,r\}$, we define
\[ \tau(\v) = (n_{\tau(1)},n_{\tau(2)},\ldots,n_{\tau(r)}). \]
We extend this to a subset $L$ of $\N_0^r$ by defining
$\tau(L) = \{\tau(\v) \mid \v\in L\}$. 
If $L = L(\c;\p_1,\ldots,\p_k)$, then $\tau(L) = (\tau(\c); \tau(\p_1),\ldots,\tau(\p_k))$,
so the property of being a linear set, or indeed an intersection of $k$ semilinear sets, is preserved by $\tau$.

We shall make significant use of the following corollary to Proposition~\ref{prop} in
Section~\ref{group}.

\begin{corollary}\label{parikh}
Let $L\subseteq w_1^*\ldots w_k^*$ be a bounded language over an alphabet $X$ with $w_i\in X^*$, 
and let $\tau$ be a permutation of $\{1,\ldots,k\}$.  
If $\tau\left(\Phi\left(L\right)\right)$ satisfies the hypothesis of Proposition~\ref{prop}, then $L$ is neither co$\cal{CF}$ nor poly-$\cal{CF}$.
\end{corollary}
\begin{proof}
Since $\tau$ preserves semilinearity, this follows immediately from 
Proposition~\ref{prop}, Theorem~\ref{Parikh} and Corollary~\ref{polyParikh}
and the fact that the class of semilinear sets is closed under intersection (Corollary~\ref{intsl})
and complementation (Proposition~\ref{compsl}).
\end{proof}

\subsection{The languages $L^{(k)}$}\label{L(k) section}

A $(k-1)$-$\cal{CF}$ language is clearly also $n$-$\cal{CF}$ for all $n\geq k$.  
In \cite{LW}, Liu and Weiner showed that the class of $k$-$\cal{CF}$ languages properly contains 
the class of $(k-1)$-$\cal{CF}$ languages, thus exhibiting an infinite heirarchy of languages 
in between the context-free and context-sensitive languages.  (They call a $k$-$\CF$ language
a `$k$-intersection language'.) Note that this implies that the $k$-$\CF$ languages are not closed under 
intersection or even under intersection with context-free languages.

There are some problems with Liu and Weiner's proof, particularly in the proof of their Theorem 10.  
In this section, we provide a more detailed proof.  
In Section~\ref{L(n,k) section}, we extend Liu and Weiner's result, but the proof of the 
special case is provided first, as it will probably aid the reader's understanding of the more
general case. 

Following Liu and Weiner, we define a sequence of languages $L^{(k)}$ and 
corresponding subsets $S^{(k)}$ of $\N_0^{2k}$.
For $k\in \N$, let $a_1,\ldots,a_{2k}$ be $2k$ distinct symbols, and define the language 
\[L^{(k)} = \{a_1^{n_1}\ldots a_k^{n_k} a_{n+1}^{n_1}\ldots a_{2k}^{n_k} \mid n_i\in \N_0\}.\]

Define $S^{(k)}$ to be $\Phi\left(L^{(k)}\right)$.  
That is, \[S^{(k)} = \{v\in \N_0^{(2k)} \mid v(i) = v(k+i) \; (1\leq i\leq k)\}.\]

The following lemma gives a condition which implies a linear set is 
not an intersection of $k-1$ stratified semilinear sets.  
The proof is assembled primarily from the proof of \cite[Lemma~4]{LW}, 
but the result is stated differently here, because
in this form it will also be useful in proving our generalisation of Liu and Weiner's result. 

\begin{lemma}\label{lemma4}
Let $S = L(\mathbf{0};P)$ be a $k$-dimensional linear subset of $\N_0^r$ such that $P$
is linearly independent over $\Q$.  
Suppose that any subset of $S$ which can be expressed as an intersection of $k-1$ 
stratified linear sets with constant vector zero has dimension at most $k-1$.  
Then $S$ is not an intersection of $k-1$ stratified semilinear sets.
\end{lemma}
\begin{proof}
If $S$ is an intersection of $k-1$ stratified semilinear sets, 
then $S$ is a finite union of intersections of $k-1$ stratified linear sets.

Let $L = \bigcap_{i=1}^{k-1} L_i$ be a subset of $S$ with each 
$L_i$ a stratified linear set.
Let $M = \bigcap_{i=1}^{k-1} L_i^{\mathbf{0}}$ and write $M = L(\mathbf{0};\p_1,\ldots,\p_m)$.
By Proposition~\ref{intlin}, there exists a finite subset $C$ of $\N_0^r$ such that 
\[ L = \bigcup_{\c_i\in C} L(\c_i; \p_1, \ldots, \p_m). \]  
For any $\c,\p\in \N_0^r$ such that $\c +n\p\in L$ for all $n\in \N_0$, we have $\p\in L$,
since $P$ is linearly independent over $\Q$.
Thus $M\subseteq S$, since ${L(\c_1;\p_1,\ldots,\p_m)\subseteq S}$.

Since $M\subseteq S$ is an intersection of $k-1$ stratified linear sets with constant zero, 
$M$ has dimension at most $k-1$ by the hypothesis of the lemma.  
Each $L(\c_i; \p_1,\ldots,\p_m)$ is a coset of $M$ and thus has the same dimension as $M$.  
Thus $L$ is a union of finitely many linear sets of dimension at most $k-1$.  
This implies that $S$ itself is a union of finitely many linear sets of dimension at most $k-1$, 
but by Proposition~\ref{dimlin}, this cannot happen since $\dim(S) = k$.
\end{proof}

\subsubsection{The new part of the proof}

This subsection contains a new proof of the result which is Theorem 10 in \cite{LW}, 
namely that $S^{(k)}$ satisfies the hypothesis of Lemma~\ref{lemma4}.  
We break most of it up into three lemmas, 
which then come together to give a relatively simple proof of the proposition itself 
(which here is Proposition~\ref{S(k)}).

\begin{lemma}\label{dimS}  Let $S = L_1\cap\ldots\cap L_k$, where each $L_i$ is a linear 
subset of $\N_0^r$ with constant vector zero and periods $P_i = \{\p_{i1},\ldots,\p_{im_i}\}$.  
For each $1\leq i\leq k$, let $\L_i = L_i^{\Q}$.  
If $\dim (S) < \dim (\L_1\cap\ldots\cap \L_k)$, then there exist $1\leq i\leq k$, $1\leq j\leq m_i$, 
such that removing $\p_{ij}$ from $P_i$ does not change the set $S$.
\end{lemma}
\begin{proof}
Suppose that $\dim (S) < \dim (\L_1\cap\ldots\cap \L_k)$ and that, for all $i$, 
removing any $\p_{ij}$ from $P_i$ changes the set $S$.  Then, for all $i, j$, 
there must exist some $\v_{ij} = \alpha_{i1}^j \p_{i1} + \ldots + \alpha_{im_i}^j \p_{im_i}\in S$ with 
$\alpha_{ij}^j\geq 1$.

Let $\{\q_1,\ldots,\q_s\}$ be a basis for $\L_1\cap\ldots\cap \L_k$.  
Since $\q_1,\ldots,\q_s\in \L_i$ for all $i$, we can write 
$\q_l = \sum_{j=1}^{m_i} \beta_{ij}^l \p_{ij}$, where $\beta_{ij}^l\in \Q$.  
Now for $1\leq i\leq k$, $1\leq j\leq m_i$, let $c_{ij} = \mathrm{min} \{\beta_{ij}^l \mid 1\leq l\leq s\}$, 
and let \[\Lambda_i = \{j \mid 1\leq j\leq m_i,\; c_{ij} <0\}.\]
Then, if $\w_i := \sum_{j\in \Lambda_i} -c_{ij}\v_{ij}$, we have $\w_i\in S$, since $\v_{ij}\in S$ 
and ${-c_{ij}\in \N}$ for all $j\in \Lambda_i$.  Each $\w_i$ can thus be expressed in $L_i$ as 
$\sum_{j=1}^{m_i} \gamma_{ij} \p_{ij}$, where ${\gamma_{ij} = \sum_{j'\in \Lambda_i} -c_{ij'} \alpha_{ij}^{j'}}$.  
Since $\w_i$ is in $S$, it also has an expression \[\w_i = \sum_{j=1}^{m_{i'}}\gamma_{i'j}^i \p_{i'j},\]
for each $i'\neq i$ in $\{1,\ldots, k\}$, where $\gamma_{i'j}^i\in \N_0$.  
For convenience, let $\gamma_{ij}^i = \gamma_{ij}$.  

Let $\w = \sum_{i=1}^k \w_i$.  Then $\w\in S$ and, for each $i$, we can write 
\[{\w = \sum_{i'=1}^k \sum_{j=1}^{m_i} \gamma_{ij}^{i'} \p_{ij}}.\]  For all $j\in \Lambda_i$, 
the coefficient of $\p_{ij}$ in this expression for $\w$ is 
\[\sum_{i'=1}^k \gamma_{ij}^{i'}\geq \gamma_{ij}^i = 
\sum_{j'\in \Lambda_i} -c_{ij'} \alpha_{ij}^{j'}\geq -c_{ij} \alpha_{ij}^j\geq -c_{ij},\] 
since $\alpha_{ij}^j\geq 1$.  Thus we have shown that for each $i$, we can express $\w$ 
in the form $\sum_{j=1}^{m_i} a_{ij} \p_{ij}$, where $a_{ij}\geq -c_{ij}$ for all $j\in \Lambda_i$.    

For any $\q_l$ in the basis for $\L_1\cap\ldots\cap \L_k$, and any $1\leq i\leq k$, we have
\[\w + \q_l = \sum_{j=1}^{m_i} (a_{ij} + \beta_{ij}^l)\p_{ij}\in L_i,\] since 
$a_{ij}+\beta_{ij}^l\geq a_{ij}+c_{ij}\geq 0$ for all $j\in \Lambda_i$, and $c_{ij}\geq 0$ for $j\notin \Lambda_i$.  
Thus $\w + \q_l\in S$ for all $1\leq l\leq s$.  Let ${M = \{\w, \w+\q_1, \ldots, \w+\q_s\}\subset S}$.  
Then $\q_1,\ldots,\q_s$ are in the subspace of $\Q^r$ generated by $M$, which is contained in~$S^{\Q}$.
Since $\{\q_1,\ldots,\q_s\}$ is a basis for $\L_1\cap\ldots\cap \L_k$, it is a linearly independent set over $\Q$. 
Thus $S^{\Q}$ has at least $s$ linearly dependent elements, contradicting 
$\dim(S) = \dim(S^{\Q}) < \dim(\L_1\cap\ldots\cap \L_k) = s$.
\end{proof}

For a stratified linear set $L\subseteq \N_0^r$, let $\rho_L$ be the symmetric relation 
on $\{1,\ldots,r\}$ given by $m\rho_L n$ if there exist non-zero $\alpha,\beta$ with
$\alpha e_m + \beta e_n\in P$.  
Define $\sim_L$ to be the reflexive and transitive closure of $\rho_L$.
This gives a partition $\Pi_L$ of $\{1,\ldots,r\}$ 
into equivalence classes under $\sim_L$.  Note that since $L$ is stratified, if $m_1 < n_1 < m_2 < n_2$, 
then at most one of $m_1\rho_L m_2$ and $n_1 \rho_L n_2$ is true.  A similar property applies to $\sim_L$:  

\begin{lemma}\label{overlap}  
Let $L\subseteq \N_0^r$ be a stratified linear set with constant vector zero.  
Then if $m_1, n_1, m_2, n_2\in \{1,\ldots,r\}$ with $m_1 < n_1 < m_2 < n_2$ and 
$m_1\not\sim_L n_1$, $m_2\not\sim_L n_2$, then $m_1\sim_L m_2$ and $n_1\sim_L n_2$ cannot both occur.
\end{lemma}
\begin{proof}
Suppose $m_1\sim_L m_2$ and $n_1\sim_L n_2$.  
Then there exist $i_1,\ldots,i_s, j_1,\ldots, j_t$ in $\{1,\ldots,r\}$ such that 
\[m_1 = i_1\rho_L i_2\rho_L\ldots\rho_L i_s=m_2  \quad \mathrm{and}  \quad 
n_1=j_1\rho_L j_2\rho_L\ldots\rho_L j_t = n_2.\]  
Let $\Lambda\in\Pi_L$ such that $m_1, m_2\in \Lambda$.  
Then since $m_1 < n_1 < m_2 < n_2$ and $n_1, n_2\notin \Lambda$, 
there must exist $k$ such that either $m_1 < j_k < m_2 < j_{k+1}$, or $j_{k+1} < m_1 < j_k < m_2$.  
Since $j_k\rho_L j_{k+1}$, this forces $i_l$ to lie between $j_k$ and $j_{k+1}$ for all $1\leq l\leq s$.  
But either $m_1(=i_1)$ or $m_2(=i_s)$ does not lie between $j_k$ and $j_{k+1}$, 
thus we have a contradiction. 
\end{proof}  

The following result gives a relationship between $\Pi_L$ and the orthogonal complement of $L^{\Q}$. 

\begin{lemma} \label{dimsimL}  
Let $L\subseteq \N_0^r$ be a stratified linear set, with $\Pi_L = \{\Lambda_1,\ldots,\Lambda_t\}$, 
and let $\L = L^{\Q}$.  Then $\L^\perp$ has a basis of the form 
$\{\x_i = \sum_{j\in \Lambda_i} \gamma_j e_j \mid i\in M\}$, where $M\subseteq \{1,\ldots,t\}$.  
In particular, $\dim(\L^\perp)=|M|\leq t$.
\end{lemma}
\begin{proof}
Let $M$ be the set of all $i\in\{1,\ldots,t\}$ such that $\x(j)\neq 0$ for some $\x\in \L^\perp$ and $j\in \Lambda_i$.  
For each $i\in M$, fix some non-zero $\x^{(i)}\in \L^\perp$ with $\x^{(i)}(j)\neq 0$ for some $j\in \Lambda_i$.  
We can write $\x^{(i)} = \sum_{j=1}^r \gamma_{ij} e_j = \sum_{s=1}^t \x^{(i)}_s$, 
where $\x^{(i)}_s = \sum_{j\in \Lambda_s} \gamma_{ij} e_j$, since $\{1,\ldots,r\}$ is the disjoint union of $\Lambda_1,\ldots,\Lambda_t$.
For $i\in M$, let $\x_i = \x^{(i)}_i$.  Then $\{\x_i \mid i\in M\}$ is a linearly independent set, 
since $\x_i\neq 0$ by the choice of $\x^{(i)}$, and $\x_i(j)=0$ for all $j\notin \Lambda_i$.

Let $P$ be the set of periods of $L$, and for $1\leq i\leq t$, let 
\[P_i = \{\alpha_me_m + \alpha_ne_n\in P \mid m,n\in \Lambda_i\},\] 
where one of $\alpha_m$ or $\alpha_n$ may be zero.  
Then $\{P_1,\ldots,P_t\}$ is a partition of $P$.  Now if $\p\in P_i$, then 
$\p\cdot \x^{(i)}_{i'} = 0$ for all $i'\neq i$, since $\x^{(i)}_{i'}(j) = 0$ for all $j\in \Lambda_i$.  
Thus $\p\cdot \x^{(i)} = \p\cdot (\x^{(i)}_1 + \ldots + \x^{(i)}_t) = \p\cdot \x^{(i)}_i = \p\cdot \x_i$.  
But $\x^{(i)}\in \L^\perp$, so $\p\cdot \x_i = 0$.  
Since also $\p\cdot \x_i=0$ for all $\p\in P_{i'}$ with $i'\neq i$, 
we have $\x_i\in \L^\perp$, for all $i\in M$.

It remains to show that $\{\x_i \mid i\in M\}$ spans $\L^\perp$.  
Recall that $\x_i = \sum_{j\in \Lambda_i} \gamma_{ij} e_j$.  
First we show that $\gamma_{ij}\neq 0$ for all $i\in M$, $j\in \Lambda_i$.  
For $i\in M$, certainly $\gamma_{im}\neq 0$ for some $m\in \Lambda_i$, since $\x_i\neq \mathbf{0}$.  
For any $n\in \Lambda_i$ there exist $m_1,\ldots,m_l\in \Lambda_i$ such that 
$m = m_1\rho_L m_2\rho_L\ldots\rho_L m_l = n$, which implies the existence of periods 
$\alpha_{m_1}e_{m_1} + \alpha_{m_2}e_{m_2},\ldots,\alpha_{m_{l-1}}e_{m_{l-1}} + \alpha_{m_l}e_{m_l}\in P_i$ 
with non-zero $\alpha_{m_j}$ for all $1\leq j\leq l$.  
Now \[\x_i\cdot (\alpha_{m_j}e_{m_j} + \alpha_{m_{j+1}} e_{m_{j+1}}) = 
\gamma_{im_j}\alpha_{m_j} + \gamma_{im_{j+1}}\alpha_{m_{j+1}} = 0\] 
for all $1\leq j\leq l-1$, since $\x_i\in \L^\perp$.  
Thus $\gamma_{im_{j+1}} = -\gamma_{im_j}\frac{\alpha_{im_j}}{\alpha_{im_{j+1}}}$ 
and so by induction $\gamma_{in} = \gamma_{im_l}\neq 0$, since $\gamma_{im} = \gamma_{i1}\neq 0$.  
Moreover, for all $n\in \Lambda_i$, the coefficient $\gamma_{in}$ is uniquely determined by $\gamma_{im}$.  
(If two different paths between $m$ and $n$ gave different values for $\gamma_{in}$, 
then our non-zero $\x_i\in \L^\perp$ could not exist.)

Finally, let $\y\in \L^\perp$ and write $\y = \sum_{j=1}^r c_j e_j = \sum_{i=1}^t \y_i$, 
where $\y_i = \sum_{j\in \Lambda_i} c_j e_j$.  
If $\y_i\neq \mathbf{0}$, then choose $j\in \Lambda_i$ with $c_j\neq 0$.  
Since $\gamma_{ij}\neq 0$, we can write $c_j = q\gamma_{ij}$, where $q\in \Q$.  
By exactly the same argument as we used for $\x_i$, we can conclude that $\p\cdot \y_i = 0$ for all $\p\in P$.  
Now for any $\alpha_je_j + \alpha_{j'}e_{j'}\in P_i$, we have 
$\y_i\cdot (\alpha_je_j + \alpha_{j'}e_{j'}) = c_j\alpha_j + c_{j'}\alpha_{j'}$, thus 
$c_{j'} = -c_j\frac{\alpha_j}{\alpha_{j'}}$.  But also $\gamma_{ij'} = -\gamma_{ij}\frac{\alpha_j}{\alpha_{j'}}$.  
Thus $c_{j'} = -q\gamma_{ij}\frac{\alpha_j}{\alpha_{j'}} = q\gamma_{ij'}$, 
and we can extend this to show that $c_n = q\gamma_{in}$ for all $n\in \Lambda_i$, thus $\y_i = q\x_i$.  
Since this applies to all $i\in M$ with $\y_i\neq \mathbf{0}$, we can conclude that $\y$ is a linear combination of 
the elements of $\{\x_i \mid i\in M\}$, and thus this set spans $\L^\perp$.
\end{proof}

We are now ready to prove Theorem 10 of \cite{LW}.

\begin{proposition}\label{S(k)}  
For $1\leq i\leq k-1$, let $L_i$ be a stratified linear set with constant vector zero, 
and let $L_1\cap\ldots\cap L_{k-1} = S\subseteq S^{(k)}$.  Then $S$ is a linear set of dimension at most $k-1$.
\end{proposition}
\begin{proof}
$S$ is a linear set with constant vector zero by Proposition~\ref{intlin}.  
Let $\L_i = L^{\Q}$ for all $1\leq i\leq k-1$, 
and let $\S = \L_1\cap\ldots\cap \L_{k-1}$.  By Lemma~\ref{dimS}, we can assume that $\dim(\S) = \dim(S)$.  
Since $S\subseteq \S$, this implies that any maximal linearly independent subset of the periods 
of $S$ is a basis for $\S$.  
Thus, since $\v(i) = \v(k+i)$ for all $\v\in S$, we also have $\v(i) = \v(k+i)$ for all $\v\in \S$.  
For all $1\leq i\leq k$, we have $e_i - e_{k+i}\in \S^\perp$, since 
$\v\cdot(e_i - e_{k+i}) = \v(i) - \v(k+i) = 0$ for all $\v\in \S$.  

Assume $\{e_i - e_{k+i} \mid 1\leq i\leq k\}$ spans $\S^\perp$, since otherwise 
$\dim(\S^\perp)\geq k+1$ and thus $\dim(\S)\leq 2k - (k+1) = k-1$.

If $\L_i^\perp\neq \{\mathbf{0}\}$, let $\Pi_{L_i} = \{\Lambda_1,\ldots,\Lambda_t\}$.  
Then, by Lemma~\ref{dimsimL}, $\L_i^\perp$ has a basis of the form $\{\x_s \mid s\in M\}$, 
where $M\subseteq \{1,\ldots,t\}$ and $\x_s = \sum_{j\in \Lambda_s} \gamma_j e_j$.  
If $s\in M$, then since $\x_s\in \L_i^\perp\subseteq \L^\perp$, we can write 
\[\x_s  = \sum_{j\in \Gamma_s} \gamma_j(e_j - e_{k+j}),\] 
where $\Gamma_s = \Lambda_s\cap \{1,\ldots,k\}$.  Certainly some $\gamma_j$ must be non-zero, 
implying $j, (k+j)\in \Lambda_s$.  Thus if $s, s'\in M$, then 
we would have some $j, (k+j)\in \Lambda_s$, $l, (k+l)\in \Lambda_{s'}$.  
But either $j<l<(k+j)<(k+l)$ or $l<j<(k+l)<(k+j)$, thus this would contradict Lemma~\ref{overlap}.  
Therefore at most one $s\in M$, and so $\dim(\L_i^\perp)\leq 1$.  This holds for all $1\leq i\leq k-1$.

But if each $\L_i^\perp$ is at most one dimensional, then 
since $\S^\perp = \L_1^\perp + \ldots + \L_{k-1}^\perp$, $\dim(\S^\perp)$ cannot exceed $k-1$, 
contradicting the fact that $e_j - e_{k+j}\in \S^\perp$ for all $1\leq j\leq k$.  
Thus our assumption that $\{e_j - e_{k+j} \mid 1\leq j\leq k\}$ spans $\S^\perp$ was false, 
and so in fact $\dim(S)\leq k-1$.
\end{proof}

\subsubsection{The rest of the proof}

\begin{theorem}\label{L(k)}{\rm \cite[Theorem~8]{LW}}
The language $L^{(k)}$ is $k$-$\cal{CF}$, but not $(k-1)$-$\cal{CF}$.  Thus, for all $k\geq 2$, 
the class of $k$-$\cal{CF}$ languages properly contains the class of $(k-1)$-$\cal{CF}$ languages.
\end{theorem}
\begin{proof}
By Corollary~\ref{polyParikh}, it suffices to show that $S^{(k)}$ is an intersection of $k$ 
but not $k-1$ stratified semilinear sets.  For $1\leq i\leq k$, define 
\[S_i = \operatorname{span} \left\{e_i + e_{k+i}, e_j \mid 1\leq j\leq 2k, 
j\notin \left\{i, k+i\right\} \right\}.\]
Then each $S_i$ is a stratified linear set and $S^{(k)} = \bigcap_{i=1}^k S_i$.
Also, $S^{(k)}$ has constant vector zero and dimension $k$, since $\{e_i + e_{k+i} \mid 1\leq i\leq k\}$ 
is a linearly independent subset which spans $S^{(k)}$.  
Hence, by Proposition~\ref{S(k)}, $S^{(k)}$ satisfies the hypothesis of Lemma~\ref{lemma4}, 
so cannot be expressed as an intersection of $k-1$ stratified semilinear sets.
\end{proof}

\subsection{The languages $L^{(n,k)}$}\label{L(n,k) section}

We can extend Theorem~\ref{L(k)} to a larger, but very similar, class of languages.  
The extended result will be used to prove that certain groups, 
for example the restricted standard wreath products $C_p\wr \Z$ (for any $p>1$), are not poly-$\cal{CF}$.

For each $n,k\in \N$, let $a_1,a_2,\ldots,a_{2nk}$ be $2nk$ distinct symbols and define
\[\begin{array}{ll}
L^{(n,k)} = \{a_1^{m_1} a_2^{m_2}\ldots a_{2nk}^{m_{2nk}} \mid & m_i\in \N_0, m_i = m_{nk+i}\; (1\leq i\leq nk),\\ 
& m_{nj+1} = m_{nj+l}\; (0\leq j\leq k-1,\; 2\leq l\leq n)\}.
\end{array}\]
For example, $L^{(2,2)} = \{ a_1^m a_2^m a_3^n a_4^n a_5^m a_6^m a_7^n a_8^n \mid m,n\in \N_0\}$.
Define $S^{(n,k)}$ to be $\Phi\left(L^{(n,k)}\right)$.  Then
\[\begin{array}{ll}
S^{(n,k)} = \{ \v\in \N_0^{2nk} \mid & \v(i) = \v(i+nk) \; (1\leq i\leq nk),\\ 
& \v(nj+1) = \v\left(nj+l\right) \; (0\leq j\leq k-1, \; 2\leq l\leq n)\}.
\end{array}\]
These sets are like $S^{(k)}$, except with each entry being repeated $n$ times.  
Thus $S^{(1,k)}$ is  just $S^{(k)}$.  For any $n\in \N$, the set $S^{(n,k)}$ has dimension $k$, 
so it is not surprising that the following result does not depend on $n$.

\begin{proposition}\label{S(n,k)}  
For $1\leq i\leq k-1$, let $L_i$ be a stratified linear set 
with constant vector zero, and let $L_1\cap\ldots\cap L_{k-1} = S\subseteq S^{(n,k)}$.  
Then $S$ is a linear set of dimension at most $k-1$.
\end{proposition}
\begin{proof}
The proof follows the idea of the proof of Proposition~\ref{S(k)}, but is a good deal more complicated.

$S$ is a linear set with constant vector zero by Proposition~\ref{intlin}.  
Let $\L_i = L_i^{\Q}$ for $1\leq i\leq k-1$, and let $\S = \L_1\cap\ldots\cap \L_{k-1}$.  
By Lemma~\ref{dimS}, we can assume that $\dim(\S) = \dim(S)$.  
Since $S\subseteq \S$, this implies that any maximal linearly independent 
subset of the periods of $S$ is a basis for $\S$.  

Thus since $\v(i) = \v(nk+i)$ for all $\v\in S$, we also have 
$\v(i) = \v(nk+i)$ for all $\v\in \S$, $1\leq i\leq nk$.  
Moreover, for all $\v\in \S$ we have $\v(nj+l) = \v(nj+l+1)$ for all $0\leq j\leq k-1$, $1\leq l\leq n-1$.
 
For all $1\leq i\leq nk$, we have $e_i - e_{nk+i}\in \S^\perp$, since, for all $\v\in \S$,
\[\v\cdot(e_i - e_{nk+i}) = \v(i) - \v(nk+i) = 0.\]  
Similarly, $e_{nj+l} - e_{nj+l+1}\in \S^\perp$ for all $0\leq j\leq k-1$ and $1\leq l\leq n-1$.  
Thus we know of $nk + (n-1)k = (2n-1)k$ linearly independent elements of $\S^\perp$. 

Assume that these $(2n-1)k$ elements form a basis of $\S^\perp$, since otherwise 
\[\dim(S) = \dim(\S) < 2nk - (2n-1)k = k,\] as we require.  
We will now derive a contradiction, using the fact that 
$\S^\perp = \L_1^\perp+\ldots+\L_{k-1}^\perp$.

For $0\leq j\leq k-1$ and $\epsilon\in \{0,1\}$, define 
\[\Delta_j^{\epsilon} = \{ n(\epsilon k +j) +l \mid 1\leq l\leq n\}\] and $\Delta_j = \Delta_j^0\cup\Delta_j^1$.  
Let $\S_j$ be the image of the projection of $\S^\perp$ onto the coordinates in $\Delta_j$.  
Since every vector in the basis of $\S^\perp$ above is contained in some $\S_j$, 
and the $\Delta_j$ are disjoint, $\S^\perp$ is the direct sum of $\S_0,\ldots,\S_{k-1}$.

Call $\x\in \S^\perp$ a \emph{$j$-bridge} if there exist $l\in \Delta_j^0$ and $l'\in \Delta_j^1$ 
such that $\x(l)$ and $\x(l')$ are both non-zero.  By extension, for $\Gamma\subseteq \{0,\ldots,k-1\}$, 
call $\x$ a \emph{$\Gamma$-bridge} if $\x$ is a $j$-bridge for all $j\in \Gamma$.   

For $0\leq j\leq k-1$, let $\Omega_j$ be the $2(n-1)$-dimensional subspace of $\S_j^\perp$ generated by 
\[\{e_{n(\epsilon k +j)+l}-e_{n(\epsilon k +j)+l+1} \mid \epsilon\in \{0,1\}, 1\leq l\leq n-1\}\]
and let $\Omega = \Omega_0 + \ldots + \Omega_{k-1}$.

Suppose that $\x$ is not a $j$-bridge for any $j$.  We will show that $\x$ must be in $\Omega$.  
Write $\x = \sum_{j=0}^{k-1} \y_j$, where $\y_j\in \S_j$.  Then no $\y_j$ is a $j$-bridge.
For any $j\in \{0,\ldots,k-1\}$, the non-zero coordinates of $\y_j$ are either 
all in $\Delta_j^0$ or all in $\Delta_j^1$, since $\y_j$ is in $\S_j$ and is not a $j$-bridge. 
For any $\v\in \S^\perp$, the sum of the entries of $\v$ is zero, 
as can be seen by considering the basis vectors of $\S^\perp$.  
Thus the subspace of $\S^\perp$ consisting of vectors whose non-zero coordinates 
all lie in $\Delta_j^\epsilon$ is spanned by 
$\{e_{n(\epsilon k +j)+l}-e_{n(\epsilon k +j)+l+1} \mid 1\leq l\leq n-1\}
\subseteq \Omega_j$, for $\epsilon\in \{0,1\}$.  
Hence $\y_j\in \Omega_j$, and since this applies for all $0\leq j\leq k-1$, we conclude 
that $\x = \sum_{j=0}^{k-1} \y_j\in \Omega$.  

If $\L_i^\perp\neq \{\mathbf{0}\}$, let $\Pi_{L_i} = \{\Lambda_1,\ldots,\Lambda_t\}$.  
Then by Lemma~\ref{dimsimL}, $\L_i^\perp$ has a basis of the form 
$B_i = \{\x_s \mid s\in M\}$, where $M\subseteq \{1,\ldots,t\}$ and $\x_s = \sum_{j\in \Lambda_s} \gamma_{sj} e_j$.  
Note that if $\x_s$ is a $j$-bridge and $s'\neq s$, $j'\neq j$, then $\x_{s'}$ cannot be a $j'$-bridge, 
since this would imply the existence of $l_1,l_2,l'_1,l'_2\in \{1,\ldots,n\}$ such that 
\[nj+l_1, n(k+j)+l_2\in \Lambda_s, \quad nj'+l'_1, n(k+j')+l'_2\in \Lambda_{s'},\] 
contradicting Lemma~\ref{overlap}.

If $B_i$ contains no $\Gamma$-bridges for any non-empty $\Gamma$, 
then every $\x_s\in B_i$ is in $\Omega$, hence $\L_i^\perp\subseteq \Omega$.
If the largest $\Gamma$ such that $\x_s$ is a $\Gamma$-bridge is a singleton $\{j\}$, 
then $B_i$ may possibly contain other $j$-bridges;  
but, as already observed, $B_i$ contains no $j'$-bridges for $j'\neq j$. 
If $\Gamma$ has at least two elements and $\x_s$ is a $\Gamma$-bridge, 
then $B_i$ contains no other $\Gamma'$-bridges, 
even for $\Gamma' = \Gamma$, since this would again imply a situation contradicting Lemma~\ref{overlap}.

Thus there is at most one $\Gamma\subseteq \{0,\ldots,k-1\}$ such that $B_i$ contains one or more $\Gamma$-bridges.  
If such $\Gamma$ exists, call it $\Gamma_i$.

For each $i$, we have $\L_i^\perp = {\cal M}_i + {\cal N}_i$, where ${\cal M}_i$ is 
the subspace generated by the $\Gamma_i$-bridge(s) and ${\cal N}_i$ is the 
subspace generated by the remaining elements of $B_i$.  

Now consider 
\[\S^\perp = \L_1^\perp + \ldots + \L_{k-1}^\perp = 
{\cal M}_1 + \ldots + {\cal M}_{k-1} + {\cal N}_1 +\ldots + {\cal N}_{k-1}.\]
Since the ${\cal N}_i$ are generated by elements which are not $\Gamma$-bridges for any non-empty $\Gamma$, 
they are all subspaces of $\Omega$.  Thus
$\S^\perp\subseteq {\cal M}_1 + \ldots + {\cal M}_{k-1} + \Omega$.

If $\Gamma_i$ contains at least two elements, then $B_i$ has a single $\Gamma_i$-bridge, 
so ${\cal M}_i$ has dimension one.  
If $\Gamma_i = \{j\}$, then even though ${\cal M}_i$ can have dimension up to $n$, 
$\Omega_j + {\cal M}_i$ has to be contained in $\S_j$, 
so can have dimension at most $2n-1$, which is one more than the dimension of $\Omega_j$.  
Thus each ${\cal M}_i$ contributes at most one extra dimension to the set 
$\Omega+{\cal M}_1 + \ldots+ {\cal M}_{k-1}$, and so 
\begin{align*}
\dim(\S^\perp) & \leq  \dim(\Omega +{\cal M}_1 + \ldots+ {\cal M}_{k-1})\\
               & \leq  2k(n-1) + k-1 = (2n-1)k -1,
\end{align*}
giving a contradiction.  Thus our assumption that $\S^\perp$ was spanned by 
${(2n-1)k}$ elements is incorrect, and so \[\dim(S) = \dim(\S)\leq 2nk - ((2n-1)k+1) = k-1. \qedhere \]
\end{proof}

\begin{corollary}\label{kS(n,k)}
A $k$-dimensional linear subset of $S^{(n,k)}$ cannot
be expressed as an intersection of $k-1$ stratified semilinear sets.
\end{corollary}
\begin{proof}
Suppose $L\subseteq S^{(n,k)}$ is $k$-dimensional and can be expressed as an
intersection of $k-1$ stratified semilinear sets.  Then we can write
$L = S_1\cup\ldots\cup S_l$, where each $S_i$ is an intersection of $k-1$ stratified linear sets.
By Proposition~\ref{intlin}, there exist finite subsets $C_i$ and $P_i$ of $\N_0^{2nk}$ such that 
$S_i = L(C_i;P_i)$ for $1\leq i\leq l$.  
By Proposition~\ref{dimlin}, there must exist $1\leq i\leq l$ and $\c\in C_i$ such
that $L(\c;P_i)$ has dimension $k$, and hence $L(\mathbf{0};P_i)$ has dimension $k$.  
Writing $S_i = \cap_{i=1}^{k-1} N_i$, where each $N_i$ is a stratified linear set, 
from Proposition~\ref{intlin} we have $L(\mathbf{0};P_i) = \cap_{i=1}^{k-1} N_i^{\mathbf{0}}$.  
But $L(\mathbf{0};P_i)$ 
is a $k$-dimensional linear subset of $S^{(n,k)}$ with constant zero, while each $N_i^{\mathbf{0}}$
is a stratified linear set, contradicting Proposition~\ref{S(n,k)}.
\end{proof}

\begin{theorem}\label{L(n,k)}
For any $k,n\in \N$, the set $S^{(n,k)}$ is not an intersection of $k-1$
stratified semilinear sets, and so the language $L^{(n,k)}$ is not $(k-1)$-$\cal{CF}$.
\end{theorem}
\begin{proof}
Recall from the proof of Proposition~\ref{S(n,k)} the notation 
\[\Delta_j = \{ nj +l \mid 1\leq l\leq n\}\cup\{ n(k +j) +l \mid 1\leq l\leq n\}.\]
For $0\leq j\leq k-1$, let $\u_j = \sum_{i\in \Delta_j} e_i$.  
Then $\{\u_j \mid 0\leq j\leq k-1\}$ is a linearly independent set which spans $S^{(n,k)}$, 
so $S^{(n,k)}$ is $k$-dimensional.  Since $S^{(n,k)}$ has constant vector zero,  
it follows from Lemma~\ref{lemma4} and Proposition~\ref{S(n,k)} that $S^{(n,k)}$ cannot be 
an intersection of $k-1$ stratified semilinear sets
and thus $L^{(n,k)}$ cannot be a $(k-1)$-$\cal{CF}$ language.
\end{proof}


\section{Poly-$\cal{CF}$ groups}\label{polyCF group}\label{group}

We begin with a simple observation, followed by our main conjecture.

\begin{observation}\label{polyprod}
The class of poly-$\CF$ groups is closed under taking finite direct products.  The direct product of
a $k_1$-$\CF$ group and a $k_2$-$\CF$ group is $(k_1+k_2)$-$\CF$.
\end{observation}
\begin{proof}
It suffices to show that the direct product of two poly-$\CF$ groups is poly-$\CF$.
Let $G_i$ be a $k_i$-$\CF$ group for $i=1,2$.  Let $A_{i1},\ldots,A_{ik_i}$ be pushdown automata 
with input alphabet $X_i$ such that a word is in $W(G_i,X_i)$ if and only if it is accepted by all $A_{ij}$.
We may assume that $X_1$ and $X_2$ are disjoint.  Now modify the automata $A_{ij}$ so that their input 
alphabet is $X = X_1\cup X_2$, but each $A_{1j}$ ignores the symbols in $X_2$ and $A_{2j}$ ignores the symbols in 
$X_1$.  
Let $h_1: X\rightarrow X_1$ be the homomorphism sending every symbol in $X_2$ to the empty word, and define
$h_2$ similarly.
Then a word $w$ in $(X\cup X^{-1})^*$ is accepted by all of the modified automata $A_{ij}$ if and only if
$h_i(w)\in W(G_i,X_i)$ for $i=1,2$.  Thus the intersection of the languages accepted by all the $A_{ij}$ is 
precisely $W(G_1\times G_2, X)$, and hence $G_1\times G_2$ is $(k_1 + k_2)$-$\CF$.
\end{proof}

Since finitely generated free groups are context-free, this implies that a 
direct product of $k$ finitely generated free groups is $k$-$\CF$.
Since the $k$-$\cal{CF}$ groups are closed under taking finite index overgroups and finitely generated 
subgroups, any finitely generated subgroup of a direct product of $k$ free groups, 
and any finite index overgroup of such a group, is $k$-$\cal{CF}$.  These are the only known $k$-$\cal{CF}$ 
groups, and we conjecture that they are the only ones.

\begin{conjecture}\label{polyconj}
Let $G$ be a finitely generated group.  Then $G$ is poly-$\cal{CF}$ if and only if $G$ is virtually 
a finitely generated subgroup of a direct product of free groups.
\end{conjecture}

This would generalise both Muller and Schupp's result on context-free groups \cite{MulSch1, MulSch2, Dun}
and the theorem of Holt, Owens and Thomas \cite{HOT}, which says that the word problem of a 
finitely generated group is an intersection of finitely many one-counter languages 
if and only if the group is virtually 
abelian.  A \emph{one-counter language} is a language recognised by a pushdown automaton with
only one stack symbol.

Note that the truth of Conjecture~\ref{polyconj} would imply that if $G$ is poly-$\CF$, then 
$W(G)$ is an intersection of finitely many deterministic context-free languages,
and hence co$W(G)$ is context-free, since the deterministic context-free languages are closed under 
complementation and the context-free languages are closed under union. 

The rest of this section is devoted to proving certain classes of groups to be 
\emph{not} poly-$\cal{CF}$.

\subsection{Some groups which are not poly-$\cal{CF}$}

Holt, Rees, R\"over and Thomas 
proved that a finitely generated nilpotent group or polycyclic group
is co$\CF$ if and only if it is virtually abelian \cite[Theorems~12 and~16]{HRRT},
and that the Baumslag-Solitar group $\BS (m,n)$ is not co$\CF$ if $m\neq \pm n$ \cite[Theorem~13]{HRRT}.
These theorems are all proved using \cite[Proposition~14]{HRRT}, which, as we have mentioned,
has a strictly weaker hypothesis than Proposition~\ref{prop}; so, with no further effort, 
we can obtain analogous results for poly-$\CF$ groups, using Corollary~\ref{parikh}.

\begin{proposition}\label{polynil}
Let $G$ be a polycyclic group or a finitely generated nilpotent group.
Then $G$ is poly-$\cal{CF}$ if and only if it is virtually abelian.
\end{proposition}
\begin{proof}
If $G$ is not virtually abelian, then the proofs of Theorems~12 (for $G$ nilpotent)
and~16 (for $G$ polycylic) in \cite{HRRT} show that there exists a regular language $R$ such that 
$\phi\left( W(G)\cap R\right)$ satisfies 
the hypothesis of Proposition~\ref{prop}, and hence $G$ is neither 
co$\CF$ nor poly-$\cal{CF}$ by Corollary~\ref{parikh}.
\end{proof}

The result for nilpotent groups was actually already obtained
by Holt, Owens and Thomas in \cite{HOT},
using what is essentially a special case of Proposition~\ref{prop}. 

The statement of Theorem~13 in \cite{HRRT} is incorrect.  It is claimed that $\BS (m,n)$ is
co$\CF$ if and only if it is virtually abelian, based on the supposition that $\BS (m,n)$ is
virtually abelian if $m = \pm n$.  We now show that if $m = \pm n$, then $\BS (m,n)$ is 
both co$\CF$ and poly-$\CF$.

\begin{proposition}\label{BSmm}
For $m\in \Z\setminus \{0\}$, the Baumslag-Solitar group $\BS (m,\pm m)$ 
is virtually a direct product of two free groups and is thus both co$\cal{CF}$
and $2$-$\cal{CF}$.
\end{proposition}
\begin{proof}
First let $G = \BS (m,m) = \left\langle x,y \mid y^{-1}x^m y = x^m\right\rangle$.
Then $x^m\in Z(G)$ and \[G/\left\langle x^m\right\rangle = \left\langle x, y \mid x^m\right\rangle = C_m * \Z.\]
Let $H/\left\langle x^m\right\rangle$ be the normal closure in 
$G/\left\langle x^m\right\rangle$ of $\left\langle y\right\rangle$.
Then \[|G/\left\langle x^m\right\rangle : H/\left\langle x^m\right\rangle| = m\] and hence ${|G:H|} = m$.  
Since $H/\left\langle x^m\right\rangle$ does not intersect any conjugate of $C_m$, by the 
Kurosh Subgroup Theorem (see for example \cite[III.3.6]{LynSch}),
$H/\left\langle x^m\right\rangle$ is the free product 
of a free group with conjugates of $\Z$, and is thus free.
Since $x^m\in Z(G)$, we have $H \cong H/\left\langle x^m\right\rangle \times \left\langle x^m\right\rangle$.  
Thus $G$ is virtually a direct product of two free groups.

Now let $G = \BS (m,-m) = \left\langle x,y \mid y^{-1}x^m y = x^{-m}\right\rangle.$
Let $K$ be the normal closure in $G$ of $\left\langle x, y^2\right\rangle$, which has 
index $2$ in $G$.  Setting $a = x$, $b = y^{-1}x^{-1}y$ and $c = y^2$ gives
\[ K = \left\langle a,b,c \mid a^m = b^m, [a^m, c]\right\rangle,\]
with $a^m\in Z(K)$.  Now take
\[ H := K/\left\langle a^m\right\rangle = \left\langle a,b,c \mid a^m = b^m = 1\right\rangle = C_m * C_m * \Z.\]
Let $\phi$ be the homomorphism from $H$ to $C_m \times C_m$ given by mapping $a$ onto a generator
of the first $C_m$ and $b$ onto a generator of the second $C_m$, and $c$ onto the identity.
Then the intersection of $\ker \phi$ with every conjugate of $\left\langle a\right\rangle$ 
and $\left\langle b\right\rangle$is trivial.  
Thus $\ker \phi$ is free, again by the Kurosh Subgroup Theorem.  
Also, ${| H : \ker \phi |} = {| C_m \times C_m |} = m^2$.  
Let $K_1$ be the preimage of $\ker \phi$ in $K$.  Since $\ker \phi$ is free and 
$\left\langle a^m\right\rangle\in Z(H)$, 
$K_1$ is isomorphic to $\ker \phi \times \left\langle a^m\right\rangle$.  Also, $K_1$ has finite index in $K$, 
and hence also in $G$, since $\ker \phi$ has finite index in $H = K/\left\langle a^m\right\rangle$.
Thus $G$ is virtually a direct product of two free groups.

Hence $G$ is $2$-$\cal{CF}$ by Observation~\ref{polyprod}, and co$\cal{CF}$ by 
the fact that the co$\CF$ groups are closed under taking finite direct products \cite[Proposition~6]{HRRT}.
\end{proof}

We can now determine which Baumslag-Solitar groups are poly-$\cal{CF}$.

\begin{proposition}\label{polyBS}
The Baumslag-Solitar group $\BS (m,n)$ is poly-$\cal{CF}$ or co$\CF$ if and only if $m = \pm n$.
\end{proposition}
\begin{proof}
The proof of Theorem~13 in \cite{HRRT} shows that if $G = \BS (m,n)$ with 
${m\neq \pm n}$, then $W(G)$ can be intersected with a regular language to give a 
sublanguage satisfying the hypothesis of Proposition~\ref{prop}, and so $W(G)$ is neither 
co$\cal{CF}$ nor poly-$\cal{CF}$ by Corollary~\ref{parikh}.
\end{proof}

\subsection{Free abelian groups and wreath products}\label{wreath-sec}

The obvious application of Proposition~\ref{L(k)} to word problems of groups is to the free abelian groups.  

\begin{lemma}\label{rankk}
A free abelian group of rank $k$ is $k$-$\cal{CF}$ but not $(k-1)$-$\cal{CF}$.
\end{lemma}
\begin{proof}
The group $\Z^k$ is a direct product of $k$ free groups, and is hence $k$-$\cal{CF}$.  
Let $\{x_1,\ldots,x_k\}$ be a generating set for $\Z^k$ and let $X_i$ denote the inverse of $x_i$.  
Consider $L = W(\Z^k)\cap (x_1^*\ldots x_k^* X_1^*\ldots X_k^*)$.  This is precisely the language 
$L^{(k)} = \{x_1^{n_1}\ldots x_k^{n_k} X_1^{n_1}\ldots X_k^{n_k} \mid n_i\in \N_0\}$ defined in 
Section \ref{L(k) section}.  Thus, by Proposition~\ref{L(k)}, $L$ is not $(k-1)$-$\cal{CF}$.  Since $L$ 
is the intersection of $W(\Z^k)$ with a regular language, this implies that $\Z^k$ is not $(k-1)$-$\cal{CF}$.
\end{proof}

The class of co$\cal{CF}$ groups is closed under taking restricted standard 
wreath products with context-free top group \cite[Theorem~10]{HRRT}.  
In contrast, we have the following result for poly-$\cal{CF}$ groups.

\begin{proposition}\label{ZwrZ}
The restricted standard wreath product $\Z\wr \Z$ is not poly-$\cal{CF}$.
\end{proposition}
\begin{proof}
Since $\Z\wr \Z$ contains free abelian subgroups of rank $k$ for all $k\in \N$,
this follows immediately from Lemma~\ref{rankk} and the fact that the poly-$\cal{CF}$
groups are closed under taking finitely generated subgroups.
\end{proof}

A further result on wreath products will be useful when we come to consider metabelian groups.  
It is our first application of Theorem~\ref{L(n,k)}.

\begin{proposition}\label{wreathp}
For any $p\in \N\setminus \{1\}$, the restricted standard wreath product $C_p\wr \Z$ is not poly-$\cal{CF}$.
\end{proposition}
\begin{proof}  
Let $G = \left\langle b\right\rangle \wr \left\langle a\right\rangle = C_p\wr \Z$, with $p>1$  
and let $A$ and $B$ be the inverses of $a$ and $b$ respectively.
For $k\in \N$, let $W_k = (A^* b a^*)^k (A^* B a^*)^k$ and let $M_k$ be the sublanguage 
of $W_k$ consisting of all those words 
\[w = (A^{m_1} b a^{n_1}) \ldots (A^{m_k} b a^{n_k}) 
(A^{m_{k+1}} B a^{n_{k+1}}) \ldots (A^{m_{2k}} B a^{n_{2k}})\] 
satisfying the following:
(i) $m_i = n_i$ for all $i$;
(ii) $n_i < m_{i+1}$ for $i\notin \{k,2k\}$.
Each of (i) and (ii) can be checked by a pushdown automaton, so $M_k$ is the intersection of 
two context-free languages and the regular language $W_k$ and is thus $2$-$\cal{CF}$.

Now let $L_k = W\left(G,\{a,b\}\right)\cap M_k$.  Then $L_k$ consists of all words of the form 
\[b^{a^{m_1}}\cdots b^{a^{m_k}} B^{a^{m_{2k+1}}}\cdots B^{a^{m_{2k}}} =_G 1,\] 
with $m_i\in \N_0$ for all $i$, and $m_i < m_{i+1}$ for $i\notin \{k,2k\}$.
Since the conjugates of $b$ in such a word are all distinct, for each $1\leq i\leq k$ we must 
have some $1\leq j\leq k$ such that $m_{k+j} = m_i$.  
But since $m_i < m_{i+1}$ and $m_{k+i} < m_{k+i+1}$ for all $1\leq i\leq k-1$, this means 
$m_i = m_{k+i}$ for all $1\leq i\leq k-1$.

When we take $\Phi(L_k)$, we can ignore the $b$'s and $B$'s, 
since these would contribute nothing to the aspects of the structure of the resulting subset 
of $\N_0^{6k}$ that interest us.  
For our purposes it is equivalent and more straightforward to consider $\Phi(L_k)$ as a subset of 
$\N_0^{4k}$, thus:
\[\Phi(L_k) = \{ (m_1,m_1,\ldots,m_k,m_k,m_1,m_1,\ldots,m_k,m_k) \mid m_i\in \N_0, m_i < m_{i+1}\}.\]
We see that $\Phi(L_k)$ is a $k$-dimensional subset of the set $S^{(2,k)}$ studied in Section~\ref{L(n,k) section}.
Thus $\Phi(L_k)$ cannot be expressed as an intersection of $k-1$ stratified semilinear sets, 
by Corollary~\ref{kS(n,k)}. 
Hence $L_k$ is not $(k-1)$-$\cal{CF}$, by Corollary~\ref{polyParikh}.  
Since $L_k$ is the intersection of $W(G)$ with a $2$-$\cal{CF}$ language, this implies that $W(G)$ 
is not $(k-3)$-$\cal{CF}$ for any $k\in \N$ and so $G$ is not poly-$\cal{CF}$.
\end{proof}

\subsection{The groups $G(\c)$}\label{Gc-sec}

The groups $G(\c)$ were defined in \cite{BH} and play an important role in the main results
of that paper, which we shall be applying in order to prove certain cases of Conjecture~\ref{polyconj}.

For $\c = (c_0,\ldots,c_s)\in \Z^{s+1}$ with $s \ge 1$, $c_0,c_s\neq 0$ and
$\gcd(c_0,\ldots,c_s) = 1$,
the group $G(\c)$ is defined by the presentation
$\left\langle a, b \mid {\cal R}_{\c}\right\rangle$, where
\[{\cal R}_{\c} = \left\{[b, b^{a^i}] \; (i\in \Z), \; 
b^{c_0}(b^a)^{c_1}\cdots(b^{a^s})^{c_s}\right\}.\]
We call such groups {\em Gc-groups}, and when we refer to the
Gc-group $G(\c)=\left\langle x,y\right\rangle$, we assume that
$\c \in \Z^{s+1}$ satisfies the above conditions, and that
$x$ replaces $a$ and $y$ replaces $b$ in the above definition of
$G(\c)$.  Note that here we depart from our usual convention of denoting the 
$i$-th component of $\c$ by $\c(i)$, as it makes the notation more pleasant.
A Gc-group is called \emph{proper} if it is not virtually abelian.

As an example,  if $\c = (-m,1)$ then $G(\c) =  \BS (1,m)$;
so the soluble Baumslag-Solitar groups  are all Gc-groups.

The main result in this section will be that a Gc-group is poly-$\cal{CF}$ 
if and only if it is virtually abelian.  We simplify the notation by setting $b_i = b^{a^i}$ for 
all $i\in \Z$, and $B = \left\langle b_i \mid i\in \Z\right\rangle$.
Since $B$ is an abelian normal subgroup of $G(\c)$ and $G(\c)/B\cong \left\langle a\right\rangle$,
we see that Gc-groups have derived length at most $2$.

\begin{lemma}\label{Gcpoly}
Let $G = G(\c)$ be a Gc-group with $|c_0| = |c_s| = 1$.
Then $G$ is polycyclic.
\end{lemma}
\begin{proof}
The relation $b_0^{\pm 1}b_1^{c_1}\cdots b_{s-1}^{c_{s-1}}b_s^{\pm 1} = 1$ implies that 
$b_0, b_{s+1}\in \langle b_1,\ldots,b_s\rangle$ and hence 
$b_i\in \langle b_1,\ldots,b_s\rangle$ for all $i$.
Hence $B = \left\langle b_1,\ldots,b_s\right\rangle$; so $G\rhd B\rhd \{1\}$ is a normal 
series for $G$ with finitely generated abelian factors and $G$ is polycyclic.
\end{proof}

Unsurprisingly, different elements of $\Z^{s+1}$ can produce isomorphic Gc-groups:

\begin{lemma}\label{reverse}
Let $G = G(\c)$, where $\c = (c_0,\ldots,c_s)$ and let $\c' = (c_s,c_{s-1},\ldots,c_0)$.
Then $G(\c)\cong G(\c')$.
\end{lemma}
\begin{proof}
Let $G = G(\c) = \left\langle a,b\right\rangle$ and let $x = a^{-1}$ and $y = b_s$.  Then 
$y^{x^i} = b_{s-i}$ for $i\in \Z$, so
$ b_0^{c_0}b_1^{c_1}\cdots b_s^{c_s} = y^{c_s} (y^x)^{c_{s-1}}\cdots (y^{x^s})^{c_0}$.\\
Hence $G(\c)\cong \left\langle x,y \right\rangle = G(\c')$. 
\end{proof}

The following proposition, proved in \cite[Proposition~2.4]{BH}, gives a useful embedding 
of a Gc-group in a semidirect product 
$\Q^s\rtimes \Z$.  

\begin{proposition}\label{Gc-embed}
Let $G= G(\c)$ be a Gc-group.
Let $\{x_1,\ldots,x_s\}$ be a basis for $\Q^s$ over $\Q$ (the rationals under addition), and let 
$\Z = \left\langle y \right\rangle$.  Let $Q = \Q^s\rtimes \Z$,  with the action of $y$ on $\Q^s$ 
being given by the (columns of the) matrix
\[A(\c) = \left( \begin{array}{cccc} 
0&\ldots&0&-c_0/c_s\\
 & & & -c_1/c_s\\
 & & & .\\
 &I_{s-1}& & .\\
 & & & .\\
 & & & -c_{s-1}/c_s 
\end{array}\right). \]
Then $G$ is isomorphic to the subgroup $\left\langle x_1,y\right\rangle$ of $Q$.
\end{proposition}

Next, we give a lemma about powers of the matrix $A(\c)$ defined in the previous proposition.

\newpage
Let $p$ be a prime.  
The \emph{$p$-adic valuation} $v_p: \Q\rightarrow \Z\cup \{\infty\}$ is given by
\begin{itemize}
\item $v_p(0) = \infty$;
\item $v_p(m/n) = d_m - d_n$ for $m,n\in \Z, \; n\neq 0$, 
where $d_k := \max \{ i\in \N_0 \mid p^i | k\}$ for all $k\in \Z$. 
\end{itemize}
We shall be concerned with powers of a prime occuring in the denominator
of various rational numbers.  Therefore, rather than $v_p$, we shall always be using 
$-v_p$, which, because of the frequency of its occurence, we shall denote by $\vp$.
Note that if $\vp (a) < \vp (b)$, then $\vp (a+b) = \vp (b)$.

The lemma is stated in slightly more generality than we require, 
as it is just as easy to prove the more general result.

\begin{lemma}\label{matrix}  
Let $M$ be a matrix of the form
\[ \left(\begin{array}{cc} 0\ldots 0 & a_1\\   & a_2\\ I_{s-1} & . \\   & .\\  & a_s\\ \end{array}\right), \]
where all $a_i\in \Q$ and at least one $a_i\notin \Z$.     
Write $M^k = (m_{ij}^{(k)})$ for $k\in \N$.  Then there exist $N\in \{1,\ldots,s\}$ 
and a prime $p$ such that, for every $k\in \N$, there exists some $i_k\leq ks$ with
$\vp (m_{Ns}^{(i_k)}) \geq k$.
\end{lemma}
\begin{proof}

Choose some $a_j\notin \Z$, and let $p$ be a prime such that $\vp (a_j) > 0$.  
Let $n = \max \{\vp (a_i) \mid 1\leq i\leq s\}$ and let $N = \max\{i \mid \vp(a_i)=n\}$.  
For $k\in \N$, denote the entry in the $N$-th row and $s$-th column of $M^k$ by $m_k$.

Note that for $k\geq 2$ and $1\leq i\leq s-1$, the $i$-th column of $M^k$ is the same as 
the $(i+1)$-th column of $M^{k-1}$.  Thus the $N$-th row of $M^k$ is 
$(\epsilon_1,\ldots,\epsilon_{s-k},m_1,\ldots,m_k)$ if $k<s$, with $\epsilon_i\in \{0,1\}$; 
and $(m_{k-s+1},\ldots,m_{k-1},m_k)$ if $k\geq s$.  For convenience, rename 
$\epsilon_1,\ldots,\epsilon_{s-k}$ as $m_{k-s+1},\ldots,m_{0}$, so that we can write the 
$N$-th row of $M^k$ in the second form in both cases.  Notice that $m_k$ is in the $N,s-i$ 
position in $M^{k+i}$.  In particular, we have $m_k$ in the $N,N$ position 
of $M^{k+s-N}$ for all $k\in \N$.

For $k\in \N$, define $i_k$ to be the minimal natural number such that $\vp (m_{i_k})\geq k$
if such a number exists, or $\infty$ otherwise.
To begin with, we have $i_1 = 1$, since $\vp(m_1) = \vp(a_N) = n\geq 1$. 
We shall show by induction on $k$ that $i_k\leq ks$ for all $k\in \N$, hence proving the lemma.

Fix $k\in \N$ and suppose that $i_k\leq ks$.  Let $j_k = i_k + s - N$ and consider $M^{j_k}$.  
The $N$-th row of this matrix is $(m_{j_k-s+1},\ldots,m_{i_k},\ldots,m_{j_k-1},m_{j_k})$. 
Note that $\vp (m_{i_k})\geq k$ and $\vp (m_i) < k$ for $1\leq i<i_k$, by the minimality of $i_k$. 
For $i\leq 0$, we have $m_i\in \{0,1\}$ and so $\vp(m_i)\in \{0,-\infty\}$.  
Thus $\vp (m_{i}) < k$ for all $i<i_k$.
Note also that $j_k + 1 = i_k + s-N +1\leq ks+s = (k+1)s$.

We may assume that $\vp (m_i)\leq k$ for all $i\leq j_k$,
since otherwise we would have $i_{k+1}\leq j_k < (k+1)s$ and we would be done.  Now 
\begin{align*}
m_{j_k+1} &= (m_{j_k-s+1},\ldots,m_{i_k},\ldots,m_{j_k})\cdot(a_1,\ldots,a_N,\ldots,a_s)\\
&= \sum_{i=1}^s m_{j_k-s+i}a_i
= \sum_{i=1}^s m_{i_k-N+i}a_i. \hspace{2cm} (*)
\end{align*}
We have $\vp (m_{i_k-N+i}a_i) = \vp (m_{i_k-N+i})+\vp(a_i)$ for $1\leq i\leq s$.  In particular, 
$\vp (m_{i_k}a_N) = \vp (m_{i_k})+n\geq k+n$. 

By the maximality of $N$, we have $\vp(a_i)<n$ for all $i>N$.  
Since also $\vp(m_{i_k-N+i})<k$ for $i<N$, we thus have $\vp(m_{i_k-N+i}a_i) < k+n$ for $i\neq N$.  
So the $N$-th term of $(*)$ has strictly
greater negative $p$-adic value than the other terms and hence
\[\vp (m_{j_k+1}) = \vp (m_{i_k}a_N)\geq k+n\geq k+1,\] therefore 
$i_{k+1}\leq j_k+1 \leq (k+1)s$, 
as required.
\end{proof}

We are now ready to prove the main result of this section.

\begin{proposition}\label{Gc}
A Gc-group is poly-$\cal{CF}$ or co$\CF$ if and only if it is virtually abelian.
\end{proposition}
\begin{proof}
Let $G = G(\c)$ be a proper Gc-group with $\c\in \Z^{s+1}$.
If $|c_0| = |c_s| = 1$, then $G$ is polycyclic and hence not poly-$\cal{CF}$
by Proposition~\ref{polynil}. 
Hence if $|c_s| = 1$, we may assume $|c_0|\neq 1$.  By Lemma~\ref{reverse}, $G$ is 
isomorphic to $G(\c')$, where $\c' = (c_s,c_{s-1},\ldots,c_0)$.  
Thus we may assume that $|c_s|\neq 1$.

By Lemma~\ref{Gc-embed}, we can identify $G$ with the subgroup $\left\langle x_1,y\right\rangle$ 
of $Q = \Q^s\rtimes \Z$, where $\{x_1,\ldots,x_s\}$ is a basis for $\Q^s$ over $\Q$,
$\Z = \left\langle y\right\rangle$, and $y$ acts on $\Q^s$ by the matrix $A(\c)$ given in the lemma.
Let $M = A(\c)$ and use the notation of Lemma~\ref{matrix} for entries of $M^k$.  
Since $|c_s|\neq 1$ and $\gcd(c_0,\ldots,c_s) = 1$, some $c_i/c_s$ for 
$0\leq i\leq s-1$ is not an integer.  Thus $M$ satisfies the hypothesis of 
Lemma~\ref{matrix}.  Hence there exist $I\in \{1,\ldots,s\}$ 
and a prime $p$ such that, for every $k\in \N$, there exists some $\iota_k\leq ks$ 
such that $\vp (m_{Is}^{(\iota_k)})$ is at least $k$.

For $k\in \N$, let 
\[ \ell_k = \min\left\{ \ell\in \N \mid \ell m_{is}^{(k)}\in \Z \;(1\leq i\leq s)\right\}.\]
This is the smallest nonnegative integer $\ell$ such that the final column of $\ell M^k$ has 
all integer entries.  We are especially interested in the matrices $M^{\iota_k}$,
and so it will be convenient to set $\lambda_k = \ell_{\iota_k}$.
Since $\vp (m_{Is}^{(\iota_k)})\geq k$, we have $\lambda_k\geq p^k$ for all $k\in \N$.

We can take an increasing sequence of natural numbers $n_1, n_2, \ldots$ such that,
for all $i\in \{1,\ldots,s\}$, the entries $m_{is}^{(\iota_{n_k})}$ 
are either nonnegative for all $k\in \N$, or negative for all $k\in \N$.
In the first case we say that $i$ is of Type 1, while in the 
second case $i$ is of Type 2.

We are now ready to define a bounded sublanguage of $W(G)$ which we can show 
to be not poly-$\cal{CF}$ using Corollary~\ref{parikh}.
Let $X = \{x_1,\ldots,x_s,y\}$ and consider the intersection of $W(G,X)$ with the
bounded context-free language
\[ L' = \cup_{k\in \N_0} (y^{-1})^k x_s^* y^k 
(x_1^{\epsilon_1})^* (x_2^{\epsilon_2})^* \ldots (x_s^{\epsilon_s})^*,\]
where $\epsilon_i = (-1)^j$ if $i$ is of Type $j$.
Let $L = \Phi \left(W(G,X)\cap L'\right)$.

The final column of $M^k$ represents the action of $y^k$ on $x_s$.  Specifically,
\[ x_s^{y^k} = x_1^{m_{1s}^{(k)}}\cdots x_I^{m_{Is}^{(k)}}\cdots x_s^{m_{ss}^{(k)}}.\]
For $\lambda\in \Z$ and $k\in \N$, the element $\left((x_s^\lambda)^{y^k}\right)^{-1}$ of $G$ can 
be expressed as a word in $(x_1^{\epsilon_1})^* (x_2^{\epsilon_2})^* \ldots (x_s^{\epsilon_s})^*$
if and only if $\ell_k | \lambda$.  
For all $k\in \N$, we thus have $(\iota_k,\lambda,\iota_k;\v)\in L$, where $\v\in \N_0^s$, 
if and only if $\ell_{\iota_k} = \lambda_k | \lambda$ and $\v(i) = \lambda |m_{is}^{(\iota_k)}|$  
for $1\leq i\leq s$.

Let $\tau$ be the permutation $(2,3)$.  Then for all $k\in \N$, 
we have $(\iota_k,\iota_k;\v)\in \tau(L)$, where $\v\in \N_0^{s+1}$, if and only if 
$\lambda_k | \v(1)$ and
$\v(i+1) = \v(1) |m_{is}^{(\iota_k)}|$ for $1\leq i\leq s$.  

For $k\in \N$, let $\a_k = ((\iota_{n_k},\iota_{n_k})$ and
let $\b_k\in \N_0^{s+1}$ with $\b_k(1) = \lambda_{n_k}$ and 
$\b_k(i+1) = \lambda_{n_k} |m_{is}^{(\iota_{n_k})}|$ for $1\leq i\leq s$.  
So $(\a_k;\b)\in \tau(L)$ if and only if $\b$
is a nonnegative integer multiple of $\b_k$.

For any $t\in \N$, there exists $N\in \N$ such that, for all $k\geq N$,
\[t\sigma(\a_k) = 2t\iota_{n_k}\leq 2tsn_k < p^{n_k}\leq \lambda_{n_k} = \b_k(1).\]
Thus, for any $k\geq N$, $\a_k$ satisfies the first two conditions of 
Proposition~\ref{prop} with respect to $t$.
We can take $k$ such that $n_k\geq t$.
For any two distinct $\b$ and $\b'$ such that $(\a;\b), (\a;\b')\in \tau(L)$, 
there are distinct $\lambda_1,\lambda_2\in \N_0$ such that
\begin{align*}
|\b(1) - \b'(1)| &= |\lambda_1\b_k(1) - \lambda_2\b_k(1)|\\
&= |\lambda_1 - \lambda_2|\lambda_{n_k}\geq p^{n_k}\geq p^t. 
\end{align*}
Since $f(t) = p^t$ is an unbounded function, this shows that $\a_k$ also satisfies 
the third condition of Proposition~\ref{prop} with respect to $t$.
Thus $\tau(L)$ is not a semilinear set and so $W(G,X)\cap L'$ is 
neither poly-$\cal{CF}$ nor co$\CF$, by Corollary~\ref{parikh}.  
Since $L'$ is context-free, this implies that $W(G,X)$ is neither poly-$\cal{CF}$ nor co$\CF$.
\end{proof}

\section{Soluble poly-$\CF$ groups}

In the case of soluble groups, Conjecture~\ref{polyconj} simplifies to

\begin{conjecture}\label{solconj}
A finitely generated soluble group is poly-$\cal{CF}$ if and only if it is virtually abelian.
\end{conjecture}

Using Theorem~\ref{solsubs} and the fact that the class of poly-$\CF$ groups is closed under taking 
finitely generated subgroups (Proposition~\ref{closureprops}), we can make some progress towards 
resolving Conjecture~\ref{solconj}.

\begin{theorem}\label{polycf-sol}
If $G$ is a finitely generated poly-$\cal{CF}$ soluble group, then one of the following must hold:
\begin{enumerate} 
\item $G$ is virtually abelian; or (possibly)
\item $G$ has a finitely generated subgroup $H$ with an infinite normal torsion subgroup $U$
such that $H/U$ is either free abelian or isomorphic to a proper
Gc-group. 
\end{enumerate} 
The second case does not occur if $G$ is metabelian or torsion-free.
\end{theorem}
\begin{proof}
By Theorem~\ref{solsubs}, if $G$ is a finitely generated soluble group which does not satisfy
(i) or (ii), then $G$ has a subgroup isomorphic to $\Z^\infty$ or a proper Gc-group.

If $G$ has a $\Z^\infty$ subgroup, then $G$ has free abelian subgroups of rank $k$ for all $k\in \N$
and so is not poly-$\CF$ by Lemma~\ref{rankk}.  If $G$ contains a proper Gc-group, then $G$ is not 
poly-$\CF$ by Proposition~\ref{Gc}.

If $G$ is torsion-free, then by definition $G$ has no non-trivial torsion subgroups.
If $G$ is metabelian, then the subgroup $H$ in the second case can be taken to be $C_p\wr \Z$ for some prime $p$,
and hence $G$ not poly-$\CF$ by Proposition~\ref{wreathp}.
\end{proof}

We conjecture that the second case does not occur at all, but have been unable
to prove this so far.

In order to complete the proof of Conjecture~\ref{solconj}, we need only show 
that a finitely generated soluble group $G$ having an infinite torsion subgroup $U$
such that $G/U$ is either free abelian or isomorphic to a proper Gc-group is 
not poly-$\cal{CF}$.

One way of approaching this which looks promising would be to show that a
poly-$\cal{CF}$ group cannot have an infinite torsion subgroup.  We know that 
context-free groups cannot have infinite torsion subgroups, because they are 
virtually free.
Actually, we conjecture something stronger, which again is true in the 
case of context-free groups.
\begin{conjecture}\label{torsionconj}
If a group $G$ is poly-$\cal{CF}$, then $G$ does not have arbitrarily large
finite subgroups.
\end{conjecture}

So far, the author's approaches towards this conjecture, from the perspective 
of automata theory, have not succeeded.  It may be that an approach using grammars
would be more fruitful.  

\subsection{An example of the undetermined case}

We give a proof of non-poly-context-freeness in a specific example of
the second case of Theorem~\ref{polycf-sol}.

If $\left\langle X \mid R\right\rangle$ is a group presentation, we denote the abelianisation of the group
with this presentation by $\mathrm{Ab}\left\langle X \mid R\right\rangle$.
This enables us to write shorter presentations for abelian groups, 
by omitting the commutators of generators from the relator set.
We call such a presentation an \emph{abelian presentation}.

\begin{proposition}\label{abc}
Let $p$ be a prime and let $G$ be the group given by the following presentation.
\[ \begin{array}{ll}
\langle a, b_i \; (i\in \Z), c_j \; (j>0) \mid & b_i^a = b_{i+1} \; (i\in \Z), \;
[b_i, b_{i+j}] = c_j \; (i\in \Z, j>0),\\
& b_i^p = c_j^p = 1 \; (i\in \Z, j>0),\; c_j\; \mathrm{central}\; (j>0) \rangle.
\end{array} \]
Then $G$ has derived length $3$ and satisfies (ii) of Theorem~\ref{polycf-sol}, and is not poly-$\cal{CF}$.
\end{proposition}
\begin{proof}
In this proof, we shall always assume that the indices on the right hand side of a 
presentation run over all available values (specified on the left hand side).
This prevents the presentations from becoming too cluttered.
With this convention, the presentation for $G$ is simplified to
\[ \left\langle b_i \; (i\in \Z), c_j \; (j>0) \mid b_i^a = b_{i+1},\;  [b_i, b_{i+j}] = c_j,
b_i^p = c_j^p = 1,\; c_j\; \mathrm{central}\right\rangle. \]  
Let $H$ be the group defined by the subpresentation
\[ \left\langle b_i \; (i\in \Z), c_j \; (j>0) \mid  [b_i, b_{i+j}] = c_j,
b_i^p = c_j^p = 1,\; c_j\; \mathrm{central}\right\rangle. \]  
Then $a$ acts on $H$ by conjugation as 
an automorphism of infinite order, so $G\cong H\rtimes \left\langle a\right\rangle$ and $G/H\cong \Z$.
Thus $G$ satisfies the second case of Theorem~\ref{polycf-sol}, with $U = H$.
Since $G\rhd H\rhd \left\langle c_j \; (j>0)\right\rangle\rhd \{1\}$ is a normal series for $G$ with abelian 
factors, $G$ has derived length at most $3$.

By standard results on `Darstellungsgruppen' (covering groups) in \cite[Chapter V.23]{Hup}, in the group $E_n$
given by the presentation
\[ \left\langle b_i \; (-n\leq i\leq n), c_{ij} \; (-n\leq i < j\leq n) \mid [b_i,b_j] = c_{ij}, b_i^p = c_{ij}^p = 1, c_{ij} \; 
\rm{central}\right\rangle, \]
the subgroup generated by all the $c_{ij}$ (which is $E'_n$) has the abelian presentation
$\mathrm{Ab}\left\langle c_{ij} \; (-n\leq i<j \leq n) \mid c_{ij}^p\right\rangle$. 

Let $E$ be the union of the ascending sequence of groups $E_1, E_2, \ldots$.  Then $E' = \cup_{i\in \N} E'_n$,
with presentation $\mathrm{Ab}\left\langle c_{ij} \; (i,j\in \Z, \; i<j) \mid c_{ij}^p\right\rangle$.
Our subgroup $H$ of $G$ is obtained from $E$ by quotienting out the subgroup 
$N:=\left\langle c_{0,j-i}c_{ij}^{-1} \mid i<j\right\rangle$
and setting $c_j = c_{0j}$ for all $j>0$.  The subgroup of $H$ generated by all the $c_j$ is isomorphic to $E'/N$,
and thus has abelian presentation 
\[ \mathrm{Ab}\left\langle c_j\; (j>0) \mid c_j^p\right\rangle. \]
In particular, all $c_j$ are non-trivial and so $H$ is not abelian, and therefore $G$ has derived length $3$.

Let $b = b_0$, $B = B_0$ and let $M_k$ be the sublanguage of
\[ W_k = (BA^*Ba^* bA^*ba^*)^k (BA^*ba^* bA^*Ba^*)^k \]
consisting of all those words 
\[ \begin{array}{l}
(BA^{m_1}Ba^{n_1} bA^{\mu_1}ba^{\nu_1})\ldots (BA^{m_k}Ba^{n_k} bA^{\mu_k}ba^{\nu_k})
(BA^{m_{k+1}}ba^{n_{k+1}} bA^{\mu_{k+1}}Ba^{\nu_{k+1}})\\ \ldots
(BA^{m_{2k}}ba^{n_{2k}} bA^{\mu_{2k}}Ba^{\nu_{2k}})
\end{array} \]
such that:
(i) $m_i = n_i = \mu_i = \nu_i$ for all $i$;
(ii) $m_i < m_{i+1}$ for $i\notin \{k,2k\}$.
The first condition can be checked by two pushdown automata, 
one checking that $m_i = n_i$ and $\mu_i = \nu_i$ for all $i$,
and the other checking that $m_i = \mu_i$ for all $i$.  The second condition can be checked by 
a single pushdown automaton.  Thus $M_k$ is $3$-$\cal{CF}$.

A word in $M_k$ is equal in $G$ to 
\[ [b, b_{m_1}]\cdots [b, b_{m_k}] [b, B_{m_{k+1}}] \cdots [b, B_{m_{2k}}] = 
c_{m_1}\cdots c_{m_k} (c_{m_{k+1}})^{-1}\ldots (c_{m_{2k}})^{-1},\]
with $m_i < m_{i+1}$ and $m_{k+i} < m_{k+i+1}$ for $1\leq i\leq k-1$.

Let $L_k = \Phi\left(W(G)\cap M_k\right)$.
As in the proof of Proposition~\ref{wreathp}, we can ignore the $b$'s and $B$'s
and take $L_k$ to be a subset of $\N_0^{8k}$.  Since the $c_{m_i}$ are 
distinct for $1\leq i\leq k$ and \[\left\langle c_j \mid j>0\right\rangle 
= \mathrm{Ab} \left\langle c_j\; (j>0) \mid c_j^p\; (j>0)\right\rangle,\]
the only way that a word in $M_k$ can be in $W(G)$ is if some $m_{k+j} = m_i$ for each $1\leq i\leq k$.
But since $m_i < m_{i+1}$ and $m_{k+i} < m_{k+i+1}$ for $1\leq i\leq k-1$, this implies that $m_i = m_{k+i}$
for $1\leq i\leq k$ and so $L_k$ is the set of all $8k$-tuples of the form
\[ (m_1,m_1,m_1,m_1,\ldots,m_k,m_k,m_k,m_k, m_1,m_1,m_1,m_1,\ldots,m_k,m_k,m_k,m_k),\]
with $m_i\in \N_0$, and $m_i < m_{i+1}$ for $1\leq i\leq k-1$.
Thus $L_k$ is a $k$-dimensional linear subset of the set $S^{(4,k)}$ introduced in Section~\ref{L(n,k) section},
and is therefore not an intersection of $k-1$ stratified semilinear sets, by Corollary~\ref{kS(n,k)}.  
By Corollary~\ref{polyParikh}, this means that $W(G)\cap M_k$ is not $(k-1)$-$\cal{CF}$.
Since $M_k$ is $3$-$\cal{CF}$, this implies that $W(G)$ is not $(k-4)$-$\cal{CF}$ for any 
$k\in \N$.  Hence $G$ is not poly-$\cal{CF}$.
\end{proof}

Quotienting out a proper subgroup of $\left\langle c_j \; (j>0)\right\rangle$ 
in the group $G$ in Proposition~\ref{abc} 
results in another group of derived length $3$
satisfying (ii) of Theorem~\ref{polycf-sol}.  We do not know how to show that such quotients are not poly-$\CF$
except in some very specific cases.

\textbf{Acknowledgements}
I am immensely grateful to my Ph.D. supervisor, Derek Holt, for many helpful and inspiring
discussions and suggestions.\\
This research was supported by a Vice Chancellor's Scholarship from the University of Warwick.

\end{document}